\documentclass[a4paper, 12pt]{amsart}

\usepackage[
a4paper ,hmargin={2.54cm,2.54cm} ,vmargin={3.17cm,3.17cm} ]{geometry}

\usepackage{amsmath, amsthm, amssymb, amsfonts, appendix, bm, comment, commutative-diagrams, enumerate, graphicx, mathrsfs, indentfirst, pgfplots, url, verbatim}
\usepackage{extarrows}
\usepackage{enumitem}

\usepackage[runin]{abstract}

\usepackage[
backend=biber ,date=year ,url=false ,doi=false
    ,eprint=true 
    ,isbn=false 
,style=alphabetic
    ,maxalphanames=4 
,maxnames=99 ,maxbibnames=99 ]{biblatex}
\usepackage[
pdfborder={0 0 0} ,colorlinks=true ,linkcolor=blue ,CJKbookmarks=true ]{hyperref}
\pgfplotsset{compat=1.17}
\newtheorem{theorem}{Theorem}[section]
\newtheorem{primedtheorem}{Theorem}

\newtheorem{corollary}{Corollary}[section]
\newtheorem{lemma}{Lemma}[section]
\newtheorem{proposition}{Proposition}[section]
\theoremstyle{definition}
\newtheorem{definition}{Definition}

\newtheorem{remark}{Remark}
\theoremstyle{plain}

\usepackage{tikz}
\usepackage{lmodern} 
\tikzset{every picture/.style={line width=0.75pt}} 

\numberwithin{equation}{section}

\title[Inverse Quasilinear Currents]{The Calderón Problem for Quasilinear Currents} 

\author[Z. Jin]{Ziyun Jin}
\address{Shanghai Center for Mathematical Sciences, Fudan University, Shanghai 200438, China. }
\email{22110180023@m.fudan.edu.cn}

\date{}

\begin{document}

\maketitle

\begin{abstract}
We study the Calder\'on problem for smooth quasilinear current laws on a bounded domain in $\mathbb{R}^n$, $n\geq3$. Under the stated structural assumptions and analyticity conditions, the interior current law is determined by local Dirichlet-to-Neumann maps, up to a partial-boundary-fixing diffeomorphism and a Hall-effect-type invisible current. The two obstructions come from the linearized inverse problem, which is solved by combining boundary determination for magnetic Schr\"odinger operators with analytic continuation in solution coordinates. Higher-order linearization and a boundary argument involving products of three solution gradients then determine the remaining Taylor coefficients.
\end{abstract}

\textbf{Keywords:} Calderón's problem, inverse problem, anisotropic conductivities, quasilinear elliptic equations, Dirichlet-to-Neumann map\par

\tableofcontents

\section{Introduction}

\textit{Electrical Impedance Tomography} seeks to determine the conductivity of a medium from voltage and current measurements on its boundary. Calder\'on \cite{MR590275} formulated the classical inverse problem for isotropic Ohmic conductors, whose current depends linearly on the electric field. In this paper we consider a non-Ohmic current law depending on position, electric potential, and electric field. In a bounded domain $\Omega\Subset\mathbb{R}^n,~n\geq3$, we denote the current density by a bundle map
\begin{equation*}
    \mathcal{J}:\mathbb{R}\oplus T^*\Omega\to T\Omega,\qquad
    (x,s,\xi)\mapsto \mathcal{J}(x,s,\xi)\in T_x\Omega.
\end{equation*}

The continuity equation of the steady current $\mathcal{J}$ gives the following \textit{Dirichlet problem}:
\begin{equation}\label{eq:quasi_gen}
    \left\{\begin{aligned}
    \partial_k(\mathcal{J}^k (x,u(x),d u(x)))  &= 0 \quad\text{in}~\Omega,\\
    u                     &= f \quad\text{on}~ \partial \Omega,
    \end{aligned}\right.
\end{equation}
where $u\in C^2(\Omega)$ denotes the steady electric potential.

\begin{remark}
    We use Einstein's summation convention throughout this article and always assume $n\geq3$. Unless otherwise specified, indices in $[1,n]\cap\mathbb{Z}$ are denoted by letters $(i,j,k,\ldots)$, those in $[1,n-1]\cap\mathbb{Z}$ by letters $(a,b,c,\ldots)$, and those in $[0,n]\cap\mathbb{Z}$ by Greek letters $(\alpha,\beta,\gamma,\ldots)$. 
\end{remark}

Local uniqueness is understood near a distinguished solution, not among all solutions of the nonlinear equation. For $0<\alpha<1$, write
\begin{equation}\label{eq:def-domain}
    B^{2,\alpha}_{\partial \Omega} (f_0,\delta_0):=\{f \in C^{2,\alpha}(\partial\Omega):
    \|f-f_0\|_{C^{2,\alpha}(\partial\Omega)}<\delta_0\}.
\end{equation}
At each distinguished solution $u^{f_0}$ of \eqref{eq:quasi_gen}, we assume \textit{uniform ellipticity} and local \textit{translation covariance}: there exist constants $C(f_0),\delta'(f_0)>0$ such that
\begin{equation}\label{condition:elliptic}
 \begin{aligned}
 \partial_{\xi_j}\mathcal{J}^k(x,u^{f_0},du^{f_0})\eta_j\eta_k
 &\geq C(f_0)|\eta|^2>0,& \quad \forall\, x\in\overline{\Omega},~\eta_j dx^j\in T^*_x\Omega\backslash\{0\},\\
 u^{f+\lambda}&\equiv u^f+\lambda,& \quad \forall\, \|f-f_0\|_{C^{2,\alpha}(\partial\Omega)}+ |\lambda|<\delta'(f_0),
 \end{aligned}
\end{equation}
while we do \textit{not} require $  \partial_{\xi_j}\mathcal{J}^k $ to be symmetric. Lemma \ref{lmm:forward} then enables us to represent the voltage-current measurement by the density-valued \textit{Dirichlet-to-Neumann (DN) map}
\begin{equation}\label{def:DN}
    \begin{aligned}
        N_\mathcal{J}: B^{2,\alpha}_{\partial\Omega}(f_0,\delta(f_0))
        &\longrightarrow C^{1,\alpha}(\partial\Omega;\Lambda^{n-1}T^*\partial\Omega),\\
        f & \longmapsto \iota^*\left(\mathcal{J}(x,u^f(x), d u^f(x)) \lrcorner (dx^1\wedge dx^2\wedge\cdots\wedge dx^n)\right),
    \end{aligned}
\end{equation}
where $ \iota:\partial \Omega\to \overline{\Omega} $ is the inclusion map and $ \delta(f_0)<\delta'(f_0) $. In fact, if we set
\begin{equation}\label{def:normal-derivative}
    \nu_k:=\iota^*\left(\partial_k \lrcorner (dx^1\wedge dx^2\wedge\cdots\wedge dx^n)\right),
\end{equation}
then $ N_\mathcal{J}(f) = \mathcal{J}^k(x,u^f(x), d u^f(x))\nu_k$. In Euclidean coordinates, $\nu_k=\mathbf{n}_k\,dS$, where $\mathbf{n} = (\mathbf{n}_1, \mathbf{n}_2, \ldots, \mathbf{n}_n)$ is the outward unit normal and $dS$ the boundary area form. We call $N_{\mathcal{J}}(f)$ the \textit{boundary current form}. Its coefficient relative to $dS$ is the normal current density $\mathbf{n}\cdot\mathcal{J}(x,u^f,du^f)$. Thus this convention is the density-valued version of the usual DN map.

Fix a nonempty open subset $\Gamma\subset\partial\Omega$. The local measurements at a constant background $s$ are
\begin{equation}\label{def:local-DN}
 N_{\mathcal{J}}^\Gamma(s,h):=N_{\mathcal{J}}(s+h)|_\Gamma,
 \qquad h\in C_c^\infty(\Gamma),\quad
 \|h\|_{C^{2,\alpha}(\partial\Omega)}<\delta(s),
\end{equation}
where $h$ is extended by zero to the rest of the boundary. The background value $s$ is prescribed on the whole boundary, and the measurements include $N_{\mathcal{J}}^\Gamma(s,0)$.

Moreover, we make the normalization assumption that
\begin{equation}\label{condition:constant-solutions}
    f_0\equiv 0,\qquad u^{0}\equiv 0,\qquad \mathcal{J}(x,0,0)\equiv 0.
\end{equation}
For constant boundary data $ s_0 $ and sufficiently small $ |s-s_0| $, \eqref{condition:elliptic} implies $ u^s-s\equiv u^{s_0}-s_0 $. Since $\mathbb{R}$ is connected and $u^0\equiv 0$ by the normalization, we have $u^s\equiv s$ for every $s\in\mathbb{R}$. Consequently, $\partial_i\mathcal{J}^i(x,s,0)=0$ for all $s$.

\begin{remark}[Normalization of the background]\label{rmk:background-normalization}
Physically, the normalization requires both zero internal potential and zero current when the boundary potential is zero.
Suppose a distinguished solution $u^{f_0}$ satisfies $u^{f_0+\lambda}\equiv u^{f_0}+\lambda$ for small constants $\lambda$. Set
\begin{equation*}
 \widehat{u}:=u-u^{f_0},\qquad
 \widehat{\mathcal{J}}(x,s,\xi):=\mathcal{J}(x,s+u^{f_0}(x),\xi+du^{f_0}(x))-\mathcal{J}(x,u^{f_0}(x),du^{f_0}(x)).
\end{equation*}
For the boundary perturbation $h=f-f_0$, the local solutions and their measurements satisfy
\begin{equation*}
 \widehat{u}^h=u^{f_0+h}-u^{f_0},\qquad
 N_{\widehat{\mathcal{J}}}(h)=N_{\mathcal{J}}(f_0+h)-N_{\mathcal{J}}(f_0).
\end{equation*}
In particular, assumption \eqref{condition:elliptic} gives $\widehat{u}^\lambda\equiv\lambda$ and  $\partial_i\widehat{\mathcal{J}}^i(x,\lambda,0)=0$ for small $\lambda$. 
This does not require $\widehat{\mathcal{J}}(x,\lambda,0)=0$. A spatially varying background shift need not preserve spatial analyticity when the potential dependence is only smooth, so the regularity condition \eqref{condition:regularity} below is imposed on the normalized current itself.

Moreover, if $\mathcal{J}=\mathcal{J}(x,\xi)$ is independent of the potential and satisfies $\mathcal{J}(x,0)=0$, then $u^0\equiv0$ satisfies the normalization \eqref{condition:constant-solutions}. In this case the equation is invariant under addition of constants, so local uniqueness also gives the translation covariance in \eqref{condition:elliptic}.
\end{remark}

The inverse problem is to recover $\mathcal{J}$ from $N_{\mathcal{J}}$. A natural obstruction is that the DN map is invariant under boundary-fixing changes of coordinates.
\begin{definition}[Boundary-fixing diffeomorphisms and current ``push-forward'']
\label{def:boundary-fixing}
A diffeomorphism $\Phi$ is \textit{boundary-fixing} if $\Phi|_{\partial \Omega}=\mathrm{id}$, or equivalently $ \Phi\circ\iota =\iota $. It is \textit{$\Gamma$-fixing} if $\Phi|_\Gamma=\mathrm{id}$. For an orientation-preserving $C^{3,\alpha}$ diffeomorphism, the ``push-forward'' of a current density is
\begin{equation}\label{def:gauge-transform}
 ( H_\Phi \mathcal{J})(y,s,\eta)
 :=\frac{D\Phi(x) \mathcal{J}(x,s,D\Phi(x)^T\eta)}
 {\det D\Phi(x)},\qquad x=\Phi^{-1}(y),
\end{equation}
in oriented coordinates, where $(D\Phi)^i{}_j=\partial_j\Phi^i$.
\end{definition}
Indeed, the transform $ H_\Phi $ is induced by the standard push-forward of differential forms. Let
\begin{equation*}
    \alpha_{\mathcal{J}}[u]:=\mathcal{J}(x,u,du)\lrcorner (dx^1\wedge dx^2\wedge\cdots\wedge dx^n)
\end{equation*}
denote the current $(n-1)$-form. Then \eqref{eq:quasi_gen} is rewritten as $d\alpha_{\mathcal{J}}[u]=0$. The DN map gives the boundary current form $N_{\mathcal{J}}(f)=\iota^*\alpha_{\mathcal{J}}[u^f]$. For a boundary-fixing diffeomorphism $\Phi$, we have
\begin{equation*}
    \alpha_{\mathcal{J}}[u]=\Phi^*\alpha_{ H_\Phi \mathcal{J}}[u\circ\Phi^{-1} ].
\end{equation*}
Since $ d\Phi^*=\Phi^*d $ and $ \iota^*\Phi^*=\iota^* $, we obtain the solution transform $u^f\mapsto u^f\circ\Phi^{-1}$, and the DN map invariance $N_{\mathcal{J}}(f) =\iota^*\alpha_{\mathcal{J}}[u^f]=\iota^*\Phi^*\alpha_{ H_\Phi \mathcal{J}}[u^f\circ\Phi^{-1} ]=N_{ H_\Phi \mathcal{J}}(f)$.

If $\Phi$ fixes only $\Gamma$, the same argument applies to $f=s+h$ with $h\in C_c^\infty(\Gamma)$: these boundary values are unchanged by composition with $\Phi^{-1}$, and the boundary current forms agree on $\Gamma$. Thus $N_{\mathcal{J}}^\Gamma=N_{H_\Phi\mathcal{J}}^\Gamma$.

The second obstruction is a current whose divergence vanishes in the domain and whose normal component vanishes on the observed boundary after substitution of any potential.
\begin{definition}[Invisible Hall-effect-type currents]\label{def:invisible-current}
A \textit{Hall-effect-type current invisible on $\Gamma$} is a current of the form
\begin{equation}\label{form:invisible}
    \delta\mathcal{J}(x,s,\xi):= \left(\mathcal{I}^k(x,s)+\varphi^{kl}(x,s)\xi_l\right)\partial_k,
\end{equation}
with differentiable coefficients $\mathcal{I}^k(x,s)$ and $\varphi^{kl}(x,s)$ satisfying
\begin{equation}\label{condition:invisible}
 \begin{gathered}
  \varphi^{kl}+\varphi^{lk}=0,\qquad
  \partial_k\varphi^{kl}+\partial_s\mathcal{I}^l=0,\qquad
  \partial_k\mathcal{I}^k=0\quad\text{in }\Omega,\\
  \nu_k\varphi^{kl}=0,\qquad \nu_k\mathcal{I}^k=0
  \quad\text{on }\Gamma.
 \end{gathered}
\end{equation}
\end{definition}
It is invisible on $\Gamma$ because, for any $C^2$ function $w$, the following identities hold in $\Omega$ and on $\Gamma$, respectively:
\begin{equation*}
    \begin{aligned}
        \partial_k (\delta\mathcal{J}^k(x,w,dw))&=\varphi^{kl}(x,w)\partial_{k}\partial_{l}w +((\partial_k\varphi^{kl})(x,w)+\partial_s\mathcal{I}^l(x,w))\partial_l w\\
        & \quad+\partial_s\varphi^{kl}(x,w)\partial_k w\partial_l w+(\partial_k\mathcal{I}^k)(x,w)=0,\qquad \\
        \nu_k (\delta\mathcal{J}^k(x,w,dw))&=\nu_k \varphi^{kl}\partial_l w+\nu_k\mathcal{I}^k=0,
    \end{aligned}
\end{equation*}
which means adding $\delta\mathcal{J}$ changes neither the equation \eqref{eq:quasi_gen} nor the measured boundary current. Consequently, $N_{\mathcal{J}}^\Gamma=N_{\mathcal{J}+\delta\mathcal{J}}^\Gamma$. Appendix \ref{app:hall-example} illustrates this antisymmetric obstruction using a linear Hall conductor in a static magnetic field, explaining the term ``Hall-effect-type''.

Our mixed regularity assumption, in real analytic coordinates on $M=\overline{\Omega}$, is
\begin{equation}\label{condition:regularity}
 \begin{aligned}
 &\mathcal{J}\in C^\infty(\mathbb{R}\oplus T^*M,TM),\\
 &\mathcal{J}(\cdot,s,\cdot)\text{ is real analytic on }T^*M,\\
 &\partial_s\mathcal{J}(\cdot,s,0)\text{ is real analytic on }M,
 \end{aligned}
 \qquad s\in\mathbb{R}.
\end{equation}
Both analyticity requirements include extension across the spatial boundary in analytic coordinates. No analyticity in $s$ or uniform analytic radius over the backgrounds is assumed. This condition includes jointly real analytic currents and ensures spatial analyticity of the first linearized coefficients at every constant background.

Under the structural assumptions and normalization stated here, coordinate changes and invisible currents are the only obstructions. We give a precise statement of the main result in the following theorem.

\begin{theorem}[Recovery of quasilinear current]\label{thm:main}
Let $M=\overline{\Omega}$, where $\Omega\subset\mathbb{R}^n$, $n\geq3$, is a bounded connected domain with nonempty real analytic boundary. Let $\mathcal{J}_1,\mathcal{J}_2:\mathbb{R}\oplus T^*M \to TM $ be bundle maps satisfying \eqref{condition:regularity}, \eqref{condition:elliptic}, and the constant-solution normalization \eqref{condition:constant-solutions}. Let $\Gamma\subset\partial M$ be nonempty and open. If, for every $s\in\mathbb{R}$, the local DN maps \eqref{def:local-DN} agree for all $h\in C_c^\infty(\Gamma)$ in a common $C^{2,\alpha}$ neighborhood of zero, then there exists an orientation-preserving $\Gamma$-fixing real analytic diffeomorphism $\Phi:M \to M $ such that
\begin{equation*}
 \mathcal{J}_1(x,s,\xi)
 = H_\Phi \mathcal{J}_2(x,s,\xi)+\delta\mathcal{J}(x,s,\xi),
\end{equation*}
where $\delta\mathcal{J}$ takes the form $\left(\mathcal{I}^k(x,s)+\varphi^{kl}(x,s)\xi_l\right)\partial_k$ as in \eqref{form:invisible}, with $\mathcal{I}$ and $\varphi$ smooth in $(x,s)$, real analytic in $x$ for each fixed $s$, and satisfying \eqref{condition:invisible}.
\end{theorem}

\medskip
\noindent\textbf{The linearized inverse problem.} The first step in recovering the quasilinear current is to solve the linearized inverse problem. For a fixed constant background $s$, write the first-order linearized coefficients and operator as
\begin{equation}\label{def:first-linearized-data}
 \begin{aligned}
 D_s^{ij}&:=\partial_{\xi_j}\mathcal{J}^i(x,s,0),
 \qquad B_s^i:=\partial_s\mathcal{J}^i(x,s,0),\\
 L_sw&:=\partial_i(D_s^{ij}\partial_jw+B_s^iw),\\
 (D^1N_{\mathcal{J}})_sh&=\nu_i(D_s^{ij}\partial_jw_s^h+B_s^iw_s^h),
 \end{aligned}
\end{equation}
where $L_sw_s^h=0$ and $w_s^h|_{\partial M}=h$. Set $\mathcal{H}_\Gamma=\mathbb{R}+C_c^\infty(\Gamma)$. The first linearized measurements on this space are determined by differentiating \eqref{def:local-DN} in the supported perturbation and in the constant background.

The linearized inverse problem itself identifies the two obstructions to recovery. We give the precise statement in the following theorem, which is proved in Section \ref{sec:first-order}.
\begin{theorem}[Recovery of linearized current]\label{thm:first-order}
Let $M$ and $\Gamma$ be as in Theorem \ref{thm:main}, and let $\mathcal{J}_1,\mathcal{J}_2$ satisfy \eqref{condition:regularity}, \eqref{condition:elliptic}, and \eqref{condition:constant-solutions}. Suppose that
\begin{equation*}
 (D^1N_{\mathcal{J}_1})_sh=(D^1N_{\mathcal{J}_2})_sh
 \quad\text{on }\Gamma,\qquad
 h\in\mathcal{H}_\Gamma,\quad s\in\mathbb{R}.
\end{equation*}
Then there is one orientation-preserving $\Gamma$-fixing real analytic diffeomorphism $\Phi:M\to M$, independent of $s$, and an antisymmetric field $\varphi_{(1)}(x,s)$, smooth in $(x,s)$ and real analytic in $x$ for each fixed $s$, such that, for every $s\in\mathbb{R}$,
\begin{align}
 \delta D_s^{ij} :=\partial_{\xi_j}(\mathcal{J}_1-H_\Phi\mathcal{J}_2)^i(x,s,0)
 &=\varphi_{(1)}^{ij},\label{eq:1st-order-coeff-a}\\
 \delta B_s^i :=\partial_s(\mathcal{J}_1-H_\Phi\mathcal{J}_2)^i(x,s,0)
 &=\partial_j\varphi_{(1)}^{ij},\label{eq:1st-order-coeff-b}\\
 \varphi_{(1)}^{ij}&=-\varphi_{(1)}^{ji},
 \qquad \nu_i\varphi_{(1)}^{ij}|_\Gamma=0.
 \label{eq:1st-order-coeff-c}
\end{align}
In particular, the (first-order) linearizations of $\mathcal{J}_1$ and $H_\Phi\mathcal{J}_2$ coincide as differential expressions at every constant background.
\end{theorem}

This theorem includes the linear current law $\mathcal{J}^i=D^{ij}(x)\xi_j+B^i(x)s$ with nonsymmetric $D$ and $\partial_iB^i=0$, as required by the constant-solution assumption. Subsection \ref{subsec:magnetic-reduction} reduces its boundary measurements to those of a magnetic Schr\"odinger operator. The following uniqueness theorem recovers the induced triple, and Subsection \ref{subsec:first-order-recovery} translates this recovery back to the current coefficients.

\begin{primedtheorem}[Recovery of the magnetic Schr\"odinger triple]\label{thm:magnetic-recovery}
Let $M$ and $\Gamma$ be as in Theorem \ref{thm:main}. For $m=1,2$, let $g_m$ be a Riemannian metric and $a_m$ a real one-form, both real analytic on $M$, satisfying $d_{g_m}^\star a_m=0$. Set $A_m=-ia_m$ and $q_m=|a_m|_{g_m}^2$. If the density-valued magnetic DN maps \eqref{def:lin-Sch-DN} satisfy
\begin{equation*}
 \Lambda_{g_1,A_1,q_1}h=\Lambda_{g_2,A_2,q_2}h
 \quad\text{on }\Gamma,\qquad h\in\mathcal{H}_\Gamma,
\end{equation*}
then there is a unique orientation-preserving $\Gamma$-fixing real analytic diffeomorphism $\Phi:M\to M$ such that
\begin{equation*}
 g_2=\Phi^*g_1,\qquad A_2=\Phi^*A_1,\qquad q_2=\Phi^*q_1.
\end{equation*}
The corresponding Dirichlet solutions satisfy $w_2^h=w_1^h\circ\Phi$ for every $h\in\mathcal{H}_\Gamma$.
\end{primedtheorem}

Subsection \ref{subsec:analytic-extension} proves Theorem \ref{thm:magnetic-recovery} by combining boundary determination, $Z$-coordinates, and global extension via Poisson embeddings. The constant solution removes scalar conjugation, and coclosedness of the drift removes the remaining conformal factor by Cauchy uniqueness. In Subsection \ref{subsec:first-order-recovery}, translation covariance makes the first solution variations independent of $s$, so the identification of the induced triples is the same at every background.

\medskip
\noindent\textbf{Recovery of higher-order terms.} After the fixed identification from Theorem \ref{thm:first-order}, Section \ref{sec:higher-order} simultaneously recovers higher current coefficients and compares the corresponding solution variations. Differentiation in $s$ gives the coefficients containing potential derivatives. The remaining coefficients, containing only gradient derivatives, enter a three-gradient integral identity after the other solution factors are fixed. The following proposition determines the resulting tensor from local boundary data. For $L=L_s$ in \eqref{def:first-linearized-data}, its formal transpose is
\begin{equation*}
 L^\star v=\partial_j(D_s^{ij}\partial_i v)-B_s^i\partial_i v,
\end{equation*}
with the coordinate density convention used in the integral below.

\begin{proposition}[Three-gradient boundary determination]\label{prop:three-gradient-boundary}
Let $M$ and $\Gamma$ be as in Theorem \ref{thm:main}, and let $L=L_s$ be the first linearized operator of a current satisfying \eqref{condition:regularity}, \eqref{condition:elliptic}, and \eqref{condition:constant-solutions}, at a fixed constant background $s$. Let $C\in\Gamma(M,(TM)^{\otimes3}\otimes\Lambda^nT^*M)$ be a real smooth contravariant three-tensor density with $C^{ijk}=C^{ikj}$. Suppose that
\begin{equation}\label{eq:int-id-3-1}
 I(C;\widetilde{w};w_1,w_2)
 :=\int_M C^{ijk}\partial_jw_1\partial_kw_2
                 \partial_i\widetilde{w}\,dx=0
\end{equation}
for all $w_1,w_2,\widetilde{w}\in C^\infty(M)$ with $Lw_1=Lw_2=0$ and $L^\star\widetilde{w}=0$, whose boundary traces lie in $C_c^\infty(\Gamma)$. Then
\begin{equation*}
 \partial_\nu^rC|_\Gamma=0,\qquad \forall\,r\in\mathbb{Z}_{\geq0}.
\end{equation*}
Here $\partial_\nu^r$ differentiates the tensor components in inward boundary normal coordinates. Vanishing of the full boundary jet is independent of that coordinate choice. If $C$ is real analytic in $M^\circ$ and up to a nonempty open subset $\Sigma\subset\Gamma$, then $C=0$ on $M$.
\end{proposition}

The proof of Proposition \ref{prop:three-gradient-boundary} in Subsection \ref{subsec:three-gradient} uses boundary quasimodes to extract a pointwise tensor identity, followed by algebraic injectivity and induction on normal derivatives. Boundary determination requires only smoothness of $C$, while the global conclusion uses analyticity. Combined with Runge approximation, the proposition recovers the remaining gradient derivatives and closes the simultaneous induction. Analyticity in $\xi$ then reconstructs the full current at each fixed $s$ in Subsection \ref{subsec:nonlinear-completion}, completing the proof of Theorem \ref{thm:main}.

\medskip
\subsection*{Relation to previous work.} For an introduction to the classical Calder\'on problem, see the textbook \emph{The Calder\'on Problem: An Introduction} by Feldman, Salo, and Uhlmann \cite{feldman2025calderon}.

\textit{Linear and linearized inverse problems.} The linear anisotropic Calder\'on problem, with
\begin{equation*}
 \mathcal{J}^i(x,\xi)=\gamma^{ij}(x)\xi_j,
 \qquad \gamma^{ij}=\gamma^{ji},
\end{equation*}
provides the geometric background. Its DN map is invariant under coordinate changes that fix the boundary. In dimensions $n\geq3$, Lee and Uhlmann \cite{lee1989determining} proved recovery of real-analytic anisotropic conductivities, modulo this natural obstruction, when the manifold is simply connected and its boundary is geodesically convex. Lassas and Uhlmann \cite{lassas2001determining} subsequently reconstructed compact connected real-analytic manifolds without topological assumptions beyond connectedness. For complete connected real-analytic manifolds of dimension $n\geq3$ with compact Wiener-regular boundary, Lassas, Taylor, and Uhlmann \cite{lassas2003dirichletneumann} showed that equality of the local DN maps on a nonempty open subset $\Gamma$ of the boundary, where the boundary and metric are analytic up to $\Gamma$, determines the manifold up to isometry. These works established the role of analyticity in passing from boundary determination to global geometric recovery. The Poisson-embedding approach later gave a flexible formulation of this continuation mechanism for Calder\'on-type problems \cite{lassas2020poisson}.

The geometric recovery builds on the solution coordinates and Poisson embeddings of Lassas, Liimatainen, and Salo \cite{lassas2020poisson,LassasMatti2022TCpf} and C\^arstea, Liimatainen, Lin, and Trad \cite{carstea2026conformalpoisson}. For the conformal Laplacian on locally conformally real analytic manifolds, \cite{LassasMatti2022TCpf} recovers the metric up to its natural conformal and coordinate gauges. For analytic metrics and an independent potential, \cite{carstea2026conformalpoisson} gives recovery from local DN data up to a coupled conformal gauge. The constant solution and the coclosedness of the drift in Theorem \ref{thm:magnetic-recovery} remove these scalar and conformal gauges, so the metric and drift are recovered modulo a $\Gamma$-fixing diffeomorphism.

A conformally transversally anisotropic (CTA) metric has the product-conformal form $g=c(x)((dx^1)^2+g_0)$ on a containing product manifold. The adjective ``simple'' refers to the transverse factor, with strictly convex boundary and a unique smoothly varying geodesic between any two points. Dos Santos Ferreira, Kenig, Salo, and Uhlmann \cite{dossantosferreira2009limiting} recover the magnetic field and electric potential on a prescribed admissible manifold, a CTA geometry with simple transverse factor. Ceki\'c \cite{cekic2020yangmills} studies recovery of Yang--Mills connections on a prescribed Riemannian base, modulo bundle gauge. Our global argument uses spatial analyticity and the coclosed drift structure to recover the unknown metric as well. The magnetic boundary determination in these articles, which is local and does not require CTA geometry or the Yang--Mills condition, nevertheless supplies an essential step of our proof.

Hamilton, Herrera, Mueller, and Von Herrmann \cite{hamilton2012complex} studied reconstruction of the complex admittivity $\sigma+i\omega\epsilon$ in dimension $ 2 $ for time-harmonic fields in the quasistatic regime, as a model for alternating-current measurements. 
In this article, we retain the steady-current setting and extend it to include Hall effects in a static magnetic field, allowing the real first-derivative matrix $\partial_\xi\mathcal{J}$ to be nonsymmetric.

\textit{Nonlinear inverse problems.} For nonlinear elliptic equations, repeated multilinear differentiation at a fixed background solution makes the nonlinearity itself a source of additional information. Early results include global uniqueness for a semilinear elliptic inverse problem by Isakov and Sylvester \cite{isakov1994global}, and the work of Sun and Uhlmann \cite{sun1997inverse} on anisotropic quasilinear media $\mathcal{J}^i(x,s,\xi)=\gamma^{ij}(x,s)\xi_j$, which transferred the then-known linear uniqueness results to potential-dependent conductivities. The higher-order linearization method grew out of the nonlinear wave-interaction framework of Kurylev, Lassas, and Uhlmann \cite{kurylev2018inverse}. In the CTA setting, Feizmohammadi and Oksanen \cite{feizmohammadi2020inverse} studied
\begin{equation*}
 -\Delta_g u+V(x,u)=0,\qquad
 V(x,s)=\sum_{k\geq1}V_k(x)\frac{s^k}{k!},
\end{equation*}
combining higher-order linearization with complex geometrical optics solutions to recover analytic semilinearities under an admissibility condition on the transverse manifold, with $V_1$ smooth and $V_1,V_2$ known a priori. Independently, Lassas, Liimatainen, Lin, and Salo \cite{lassas2021inverse} considered
\begin{equation*}
 \Delta_g u+q(x)u^m=0,
\end{equation*}
obtaining simultaneous geometric and coefficient recovery in dimension two for $m\geq2$, and potential recovery on transversally anisotropic manifolds with arbitrary transverse geometry in dimensions $n\geq3$ for $m\geq3$.

For genuinely quasilinear conductivities with dependence on the potential and/or its gradient, Mu\~noz and Uhlmann \cite{munoz2020calderon} and, subsequently, Shankar \cite{shankar2020recovering} obtained recovery results for spatially homogeneous scalar current laws of the form
\begin{equation*}
 \mathcal{J}(x,s,\xi)=a(s,\xi)\xi,
\end{equation*}
under analytic or structural assumptions. Shankar also treated nonhomogeneous conductivities when the spatially dependent part is prescribed. C\^arstea and Feizmohammadi \cite{carstea2021inverse} treated anisotropic gradient nonlinearities with a prescribed Laplacian linear part, of the form
\begin{equation*}
 \mathcal{J}^i(x,s,\xi)=\delta^{ij} \xi_j+
 \sum_{k=2}^{N}\mathcal{J}^i_{k;j_1\cdots j_k}(x)
                  \xi_{j_1}\cdots\xi_{j_k}
 +\mathcal{R}^i(x,s,\xi),\qquad
 \mathcal{R}=O(|\xi|^{N+1}).
\end{equation*}
They recovered the displayed anisotropic coefficients using higher-order linearization and Gaussian quasimodes. C\^arstea, Feizmohammadi, Kian, Krupchyk, and Uhlmann \cite{carstea2021calderon} proved global uniqueness for the spatially dependent isotropic law
\begin{equation*}
 \mathcal{J}(x,s,\xi)=a(x,s,\xi)\xi
\end{equation*}
in dimensions $n\geq3$ under holomorphic or real-analytic dependence assumptions, using plane-concentrating complex geometrical optics solutions. In related work, C\^arstea and Feizmohammadi \cite{carstea2023density} established a density theorem for tensor products of at least three gradients of harmonic functions and applied it to a coupled quasilinear system. Proposition \ref{prop:three-gradient-boundary} gives an analogous boundary determination result, using solutions of the first linearized operator $L$ and its transpose in place of harmonic functions. 

Recently, Liimatainen and Wu \cite{liimatainen2024calderon} resolved the corresponding scalar conductivity problem in dimension two in a holomorphic small-data class, using Bukhgeim-type solutions and stationary phase and assuming that the two conductivities agree to infinite order on the boundary. Chen and Jin \cite{chen2026calderon} recovered all $(s,\xi)$-derivatives of the scalar factor $ a(x,s,\xi) $ at $\xi=0$ for $ \mathcal{J}^i(x,s,\xi)=a(x,s,\xi)\sqrt{\det g(x)}\,g^{ij}(x)\xi_j $, where $g$ is a fixed simple CTA metric.

All of the conductivity models just discussed have a symmetric derivative $\partial_\xi\mathcal{J}(x,s,0)$ at constant backgrounds. Theorem \ref{thm:main} treats current laws satisfying \eqref{condition:regularity} in the stated translation-covariant class, without prescribing an isotropic factorization. It recovers the current law at all constant backgrounds modulo two unavoidable obstructions: an analytic change of coordinates fixing $\Gamma$ and the Hall-effect-type current invisible on $\Gamma$ in \eqref{form:invisible}--\eqref{condition:invisible}.

\medskip
\noindent\textbf{Acknowledgements.} The author thanks Tony Liimatainen and Mikko Salo for helpful discussions on the analytic continuation arguments needed to pass from boundary determination to recovery in the interior. The author also thanks Xi Chen and Gunther Uhlmann for discussions on quasilinear Calderón problems, during which Uhlmann posed an early version of the problem studied in this article. The author was supported in part by CSC grant 202506100141 and NSFC No. 12571451.

\medskip
\noindent\textbf{AI disclaimer.} ChatGPT (OpenAI) and Gemini (Google) assisted with proof adaptation, example construction, and manuscript revision.

\section{Preliminaries}\label{sec:preliminaries}

We work on a connected compact oriented real analytic manifold $M$ of dimension $n\geq3$, with nonempty real analytic boundary.

\subsection{Forward problems}
We first justify the local solution map and its differentiability by applying the implicit function theorem for Banach spaces. In the function-space statements below, we divide the density-valued equation by a fixed positive analytic density. We write $M^\circ$ for the interior, $C^{2,\alpha}_0(M)$ for functions with zero Dirichlet trace, and $L^\star$ for the formal adjoint with respect to this density.

\begin{lemma}[Local solutions and smooth dependence]\label{lmm:forward}
Let $M$ be a connected compact real analytic manifold with nonempty smooth boundary, let $0<\alpha<1$, and let a $C^\infty$ current $\mathcal{J}$ admit a solution $u^{f_0}\in C^{2,\alpha}(M)$ of \eqref{eq:quasi_gen}. Assume uniform ellipticity at $u^{f_0}$ and that the first-order linearized zero-Dirichlet operator
\begin{equation*}
    \begin{aligned}
        L_0:C^{2,\alpha}_0(M)\longrightarrow C^\alpha(M),&\qquad
        L_0w=\rho^{-1}\partial_i\left(a^{ij}\partial_jw+b^iw\right),\\
        a^{ij}:=\partial_{\xi_j}\mathcal{J}^i(x,u^{f_0},du^{f_0}),&\qquad
        b^i:=\partial_s\mathcal{J}^i(x,u^{f_0},du^{f_0}),
    \end{aligned}
\end{equation*}
is an isomorphism, where $\rho$ is some positive density function in local coordinates. This isomorphism condition follows in particular if $u^{f_0}+t$ solves the equation for every sufficiently small real $t$.

There are neighborhoods $U$ of $f_0$ in $C^{2,\alpha}(\partial M)$ and $V$ of $u^{f_0}$ in $C^{2,\alpha}(M)$ such that each $f\in U$ has exactly one solution $S(f)=u^f\in V$. Both
\begin{equation*}
 S:U\longrightarrow C^{2,\alpha}(M),\qquad
 N_{\mathcal{J}}:U\longrightarrow
 C^{1,\alpha}(\partial M;\Lambda^{n-1}T^*\partial M)
\end{equation*}
are $C^\infty$. For every finite $K$, their $K$-th derivatives are continuous symmetric multilinear maps on $(C^{2,\alpha}(\partial M))^K$. At constant backgrounds, solution derivatives in smooth boundary directions are smooth up to the boundary.
\end{lemma}
\begin{proof}
We use the function-space framework of \cite[Theorem B.1]{carstea2021calderon}, applied componentwise to the current. Set $X=C^{2,\alpha}(\partial M)$, $Y=C^{2,\alpha}(M)$, and $Z=X\times C^\alpha(M)$, and define
\begin{equation*}
 \mathscr{C}(u):=\mathcal{J}(x,u,du),\qquad
 F(f,u):=\left(f-u|_{\partial M},\rho^{-1}\partial_i\mathscr{C}^i(u)\right).
\end{equation*}
The chain rule gives for every $r\geq1$
\begin{equation*}
 D^r\mathscr{C}(u)[v_1,\ldots,v_r]
 =\partial_{(s,\xi)}^r\mathcal{J}(x,u,du)
       [(v_1,dv_1),\ldots,(v_r,dv_r)].
\end{equation*}
On a sufficiently small neighborhood $V$ of $u^{f_0}$ in $Y$, the evaluated fiber variables stay in a fixed compact set. Boundedness of $u\mapsto(u,du)$ into $C^{1,\alpha}$ and the H\"older product estimates give
\begin{equation*}
 \|D^r\mathscr{C}(u)[v_1,\ldots,v_r]\|_{C^{1,\alpha}}
 \leq C_{r,V}\prod_{j=1}^r\|v_j\|_{C^{2,\alpha}},
 \qquad u\in V.
\end{equation*}
Real-variable Taylor formulae and the same estimates control the remainders and continuity of the derivatives, using only finitely many derivatives of the smooth current at each order. A finite atlas therefore gives a $C^\infty$ map $F$ near $(f_0,u^{f_0})$.

Let $Lw=\rho^{-1}\partial_i(a^{ij}\partial_jw+b^iw)$ on $Y$, and choose a bounded extension operator $E:X\to Y$ as a right inverse of the Dirichlet trace. The derivative and its inverse are
\begin{equation*}
 \begin{aligned}
 D_uF(f_0,u^{f_0})w&=(-w|_{\partial M},Lw),\\
 \bigl(D_uF(f_0,u^{f_0})\bigr)^{-1}(h,G)
 &=-Eh+L_0^{-1}(G+L(Eh)).
 \end{aligned}
\end{equation*}
The Banach implicit function theorem \cite[Theorem 15.1 and Corollary 15.1]{deimling1985nonlinear} gives the unique local solution map $S\in C^\infty(U;Y)$. Its $C^k$ version gives $S\in C^k$ whenever $F\in C^k$, for finite $k\geq1$.

Since $N_{\mathcal{J}}(f)=\nu_i\mathscr{C}^i(S(f))|_{\partial M}$, composition and the trace map give the same parameter regularity for the DN map. For each fixed $K$, locally in $f$,
\begin{equation*}
 \|D^KS(f)[h_1,\ldots,h_K]\|_{C^{2,\alpha}}
 \leq C_K\prod_{j=1}^K\|h_j\|_{C^{2,\alpha}}.
\end{equation*}
Here $C_K$ may depend on the background neighborhood, ellipticity and coefficient norms. No uniformity over all real backgrounds is asserted. Differentiating $S(f)|_{\partial M}=f$ gives the prescribed trace at first order and zero trace at higher orders. At a constant background, successive variations solve the same elliptic equation with smooth coefficients and sources formed from lower variations. Schauder regularity and induction give spatial smoothness for smooth boundary directions.

In particular, \eqref{condition:elliptic} is sufficient for invertibility without presupposing differentiability of a solution map. If $u^{f_0}+t$ is a solution for small $t$, direct differentiation of this explicitly given affine family gives
\begin{equation*}
 L1=\rho^{-1}\partial_i b^i=0.
\end{equation*}
Since $a,b\in C^{1,\alpha}$, expansion of $L_0$ has $C^\alpha$ coefficients, uniformly positive definite symmetric principal part, and vanishing zeroth-order coefficient. The maximum principle gives uniqueness, and the classical Dirichlet theory gives existence and the Schauder bound. See \cite[Theorem 6.14]{gilbarg2001elliptic} in a $ C^{2,\alpha}  $ domain. On $M$ the same conclusion follows from the local boundary and interior Schauder estimates, the maximum principle, and the method of continuity from an auxiliary Laplace operator. This proves the asserted isomorphism.
\end{proof}


\begin{definition}[Admissible backgrounds]\label{def:admissible-background}
For a current $\mathcal{J}$, a boundary value $f_0$ together with a distinguished solution $u^{f_0}$ is an \textit{admissible background} if it satisfies the hypotheses of Lemma \ref{lmm:forward}. The associated solution map is the unique map furnished by that lemma. Under the hypotheses of Theorem \ref{thm:main}, every constant background $f_0\equiv s$, $u^{f_0}\equiv s$ is admissible.
\end{definition}

\subsection{Runge approximation and Poisson embeddings}
\label{subsec:poisson-tools}

We use the partial-boundary Runge theorem for general second-order elliptic operators in \cite[Proposition~A.2]{lassas2020poisson}. Its finite-dimensional jet argument, used for harmonic functions in \cite[Proposition~A.5]{lassas2020poisson}, gives the following analytic one-jet version. A \textit{one-jet} at $x$ is the pair $(w(x),dw(x))$. All operators act on functions. A density-valued equation may first be divided by a fixed positive density. Appendix \ref{app:poisson-continuation} records the adaptations needed for the real elliptic operators used here.

The Runge approximation theorem for one-jets is a key tool in the Poisson-embedding approach to Calder\'on-type problems. It allows us to approximate arbitrary values and first derivatives of solutions at interior points by solutions with boundary data supported in an open subset of the boundary.

\begin{lemma}[Runge approximation for one-jets]\label{lmm:Runge-jets}
Let $M$ be a compact connected real analytic manifold of dimension $n\geq3$ with nonempty real analytic boundary. Let $L$ be a real second-order uniformly elliptic operator with real analytic coefficients on $M$, and suppose its zero-Dirichlet realization is invertible. Let $\Gamma\subset\partial M$ be nonempty and open. At any finite collection of distinct interior points, the solutions $Lw=0$ with real boundary data in $C_c^\infty(\Gamma)$ realize arbitrary values and first derivatives. In particular, their differentials span $T_x^*M$ at every $x\in M^\circ$.
\end{lemma}
The reduction to the cited Runge theorem is given in Appendix \ref{app:runge-jets}.

\begin{lemma}[Local analytic boundary regularity]
\label{lmm:analytic-boundary}
Let $\Gamma\subset\partial M$ be relatively open, and suppose that $\partial M$ is real analytic near each point of $\Gamma$. Let $L$ be a second-order elliptic operator with real analytic coefficients and definite real principal symbol in a neighborhood of $\Gamma$. A solution of $Lw=0$ that is smooth up to $\Gamma$ and has an analytic Dirichlet trace on $\Gamma$ extends real analytically across the boundary near each point of $\Gamma$.
\end{lemma}
This is the local analytic Dirichlet regularity theorem of Morrey \cite[Theorem~5.7.1$'$]{morrey1966multiple}, applied after analytically flattening the boundary and subtracting an analytic extension of the trace. The definite principal symbol makes the Dirichlet boundary noncharacteristic and satisfies the elliptic complementing condition.

In anisotropic Calder\'on problems, one must identify the unknown geometry up to a boundary-fixing diffeomorphism. Embedding the manifold into a space of boundary functionals provides a way to construct this identification from solutions. This is the reason why we introduce the solution coordinates (or $ Z $-coordinates) for local determination and the Poisson embedding for global identification.

The Poisson embedding below is the full-boundary, fixed-Sobolev-space version of \cite[Definition~2.1]{lassas2020poisson}. Solution coordinates are used in \cite[Lemma~2.4]{lassas2020poisson}, and the boundary construction in Lemma \ref{lmm:formal-to-collar} follows the $Z$-coordinates of \cite[Propositions~3.2--3.3]{LassasMatti2022TCpf}. For a recent account of quotient solution coordinates for $\Delta_g+q$, see C\^arstea, Liimatainen, Lin, and Trad \cite[Section~3.4]{carstea2026conformalpoisson}.
\begin{definition}[Poisson embeddings and solution coordinates]
\label{def:poisson-embedding}
Let $M$ and $L$ satisfy the hypotheses of Lemma \ref{lmm:Runge-jets}, assume $L1=0$, and fix $r>(n-1)/2+2$. Write $w^h$ for the solution of $Lw^h=0$ with $w^h|_{\partial M}=h$, initially for $h\in C^\infty(\partial M)$. Regard $H^{-r}(\partial M)$ as the continuous dual of $H^r(\partial M)$, with Sobolev norms defined using any fixed smooth auxiliary metric. The Poisson embedding associates to an interior point the functional that evaluates all Dirichlet solutions there. Its extension to boundary points, where evaluation becomes the Dirac functional $\delta_x$, is the \textit{compactified Poisson embedding}:
\begin{equation}\label{eq:compactified-poisson-map}
 \overline{\mathscr{P}}_L:M\longrightarrow
 H^{-r}(\partial M),\qquad
 \langle\overline{\mathscr{P}}_L(x),h\rangle= w^h(x).
\end{equation}

The restriction $\mathscr{P}_L:=\overline{\mathscr{P}}_L|_{M^\circ}$ is called the \textit{Poisson embedding}. The embedding properties of both maps are established in Lemma \ref{lmm:compactified-poisson}.

For local data, put $E_\Gamma=\overline{C_c^\infty(\Gamma)}^{H^r(\partial M)}$ and define $\mathscr{P}_L^\Gamma(x)=\overline{\mathscr{P}}_L(x)|_{E_\Gamma}\in E_\Gamma^*$. Its restriction to $M^\circ\cup\Gamma$ is the local Poisson embedding. Its continuous extension to $M$ vanishes on $\partial M\setminus\Gamma$.

A \textit{solution coordinate map} is a map $Z=(w^{h_1},\ldots,w^{h_n})$ whose component differentials are linearly independent at the point under consideration. In the ratio construction of \cite[Proposition~3.2]{LassasMatti2022TCpf}, the coordinates have the form $w^{h_j}/w^1$. Here $L1=0$ and Dirichlet uniqueness give $w^1\equiv1$, so the denominator disappears. The inverse function theorem then gives a local coordinate chart. The boundary construction and its analytic regularity are verified for our drift operator below.
\end{definition}

The next lemma adapts \cite[Propositions~2.1--2.3 and Appendix~B]{lassas2020poisson}. We use $L1=0$ for boundary separation and include the uniform estimates that give analyticity in one fixed Sobolev dual space. A brief proof is given in Appendix \ref{app:poisson-estimates}.

\begin{lemma}[Compactified Poisson embedding]\label{lmm:compactified-poisson}
Let $L$ satisfy the hypotheses of Lemma \ref{lmm:Runge-jets} and $L1=0$, and let $r>(n-1)/2+2$. The map \eqref{eq:compactified-poisson-map} is continuous, injective on $M$, real analytic on $M^\circ$, and an immersion there. Thus $\overline{\mathscr{P}}_L$ is a topological embedding of $M$, and its restriction $\mathscr{P}_L$ is a real analytic embedding of $M^\circ$. The local map $\mathscr{P}_L^\Gamma$ is a topological embedding of $M^\circ\cup\Gamma$ and a real analytic immersion on $M^\circ$.
\end{lemma}

The analytic continuation argument of \cite[Theorem 3.3]{lassas2020poisson} uses separation by solutions, solution coordinates, and the analytic identity theorem. We record its local-data version for the real elliptic operators with $L1=0$ used here. The adaptation from harmonic functions uses Lemmas \ref{lmm:Runge-jets} and \ref{lmm:compactified-poisson}. It is not a direct application of the harmonic case of that theorem.

\begin{lemma}[Global extension via Poisson embeddings]\label{lmm:poisson-continuation}
Let $L_m$, $m=1,2$, satisfy the hypotheses of Lemma \ref{lmm:compactified-poisson}, and write $w_m^h$ for their Dirichlet solutions. Suppose that, for each $p\in\Gamma$, there are neighborhoods $V_{m,p}$ of $p$ in $M$ and a $\Gamma$-fixing real analytic diffeomorphism $F_p:V_{2,p}\to V_{1,p}$ such that
\begin{equation*}
 w_2^h=w_1^h\circ F_p\quad\text{on }V_{2,p}
 \quad\text{for every }h\in C_c^\infty(\Gamma).
\end{equation*}
Then these local identifications extend to a unique $\Gamma$-fixing real analytic diffeomorphism $\Phi:M^\circ\cup\Gamma\to M^\circ\cup\Gamma$ with the same solution identity. If the local identifications preserve orientation, then so does $\Phi$.
\end{lemma}
The proof, including control of the inverse coordinate homeomorphism at limit points, is given in Appendix \ref{app:poisson-globalization}.

Equality of the local nonlinear measurements for every $s$ determines the linearized measurements on $\mathcal{H}_\Gamma=\mathbb{R}+C_c^\infty(\Gamma)$, as used in Theorem \ref{thm:first-order}, since
\begin{equation*}
 \partial_s^a D_h^bN_{\mathcal{J}}^\Gamma(s,0)[h_1,\ldots,h_b]
 =\left.(D^{a+b}N_{\mathcal{J}})_s[\underbrace{1,\ldots,1}_{a},h_1,\ldots,h_b]\right|_\Gamma.
\end{equation*}
Here $a,b\in\mathbb{Z}_{\geq0}$, $a+b\geq1$, and $h_j\in C_c^\infty(\Gamma)$. The identity follows by first differentiating the boundary perturbations at each $s$ and then differentiating the resulting smooth identities in $s$.

\subsection{Magnetic operators and boundary jets}
\label{subsec:magnetic-tools}

For a smooth Riemannian metric $g$, a complex one-form $A$, and a complex function $q$, we use
\begin{equation*}
 d_g^\star:=(-1)^k\star_g^{-1}d\star_g
 \quad\text{on $k$-forms},\qquad -\Delta_g=d_g^\star d,
 \qquad d_A=d+iA\wedge,
\end{equation*}
where $\star_g$ is the \textit{Hodge star} and $d_g^\star$ the \textit{codifferential} (the formal adjoint of $d$). We write $\mathcal{P}_{g,A,q}=d_{\overline{A}}^\star d_A+q$ for the \textit{magnetic Schr\"odinger operator}, where the adjoint of $d_{\overline{A}}$ is taken in the usual Hermitian $L^2$ pairing. Our convention for complex magnetic potentials is that of \cite[Section~1]{dossantosferreira2009limiting}:
\begin{equation*}
 d_{\overline{A}}^\star d_Aw
 =-\frac{1}{\sqrt{\det g}}(\partial_j+iA_j)
 \left(\sqrt{\det g}\,g^{jk}(\partial_k+iA_k)w\right).
\end{equation*}
Thus
\begin{equation*}
 d_{\overline{A}}^\star d_Aw
 =-\Delta_gw-2iA^j\partial_jw
 +\left(i\,d_g^\star A+g^{jk}A_jA_k\right)w,
\end{equation*}
where $g^{jk}A_jA_k$ is complex bilinear. When the zero-Dirichlet realization is invertible, let $w^h$ solve $\mathcal{P}_{g,A,q}w^h=0$, $w^h|_{\partial M}=h$. The \textit{density-valued magnetic Schr\"odinger map} is
\begin{equation}\label{def:lin-Sch-DN}
 \Lambda_{g,A,q}h
 :=\iota^*((d_Aw^h)^\sharp\lrcorner d\operatorname{Vol}_g)
 =\nu_i\sqrt{\det g}\,g^{ij}(\partial_j+iA_j)w^h.
\end{equation}
Here $\sharp$ raises a covector index using $g$, and $\lrcorner$ denotes contraction with the volume form. The $\nu_i$ retain the form-valued convention of \eqref{def:normal-derivative}. Dividing this map by the boundary volume density gives the magnetic normal derivative used in the cited literature.

The following boundary determination lemma packages \cite[Theorem 8.4]{dossantosferreira2009limiting} with the corrected conformal normalization of \cite[Lemma 2.1(b)]{LassasMatti2022TCpf}. For a recent treatment of conformal boundary determination for $\Delta_g+q$, see C\^arstea, Liimatainen, Lin, and Trad \cite[Section~2.2]{carstea2026conformalpoisson}. The corresponding normalized unitary vector-bundle result is \cite[Theorem 3.4]{cekic2020yangmills}, whereas here we use the complex convention above, which includes purely imaginary $A$.

\begin{lemma}[Magnetic boundary jets]\label{lmm:magnetic-boundary-jets}
Let $M$ be a compact smooth oriented manifold of dimension $n\geq3$ with nonempty smooth boundary. Let $g_m$ be smooth Riemannian metrics on $M$, $A_m$ smooth purely imaginary one-forms, and $q_m$ smooth complex functions, $m=1,2$. Assume that the zero-Dirichlet realizations of $\mathcal{P}_m=\mathcal{P}_{g_m,A_m,q_m}$ are invertible, that the induced boundary metrics agree on $\Gamma$, and that $\Lambda_{g_1,A_1,q_1}h=\Lambda_{g_2,A_2,q_2}h$ on $\Gamma$ for every $h\in\mathcal{H}_\Gamma$. Let $w_m^h$ denote their Dirichlet solutions. Then at every point of $\Gamma$ there are a common smooth collar $\mathcal{C}=U\times[0,\varepsilon)$, with $U\subset\Gamma$ an open boundary neighborhood of the point, boundary-fixing collar maps $\chi_m:\mathcal{C}\to M$, and smooth positive functions $\kappa_0,\lambda_0$ on $\mathcal{C}$ such that, coefficientwise,
\begin{align}
 \kappa_0^{-1}\mathcal{P}_2^{\mathcal{C}}\kappa_0
 -\lambda_0\mathcal{P}_1^{\mathcal{C}}
 &=O(t^\infty),\label{eq:formal-operator-package}\\
 U_2^h-\kappa_0U_1^h&=O(t^\infty)
 \quad\text{for every }h\in\mathcal{H}_\Gamma,\label{eq:formal-solution-package}\\
 \kappa_0|_{t=0}=\lambda_0|_{t=0}&=1,\qquad
 d\lambda_0|_{t=0}=0.\label{eq:formal-boundary-normalization}
\end{align}
Here $U_m^h=w_m^h\circ\chi_m$, and $\mathcal{P}_m^{\mathcal{C}}$ denotes the expression pulled back by $\chi_m$: explicitly,
\begin{equation*}
    \mathcal{P}_m^{\mathcal{C}} v:=(\chi_m^*\mathcal{P}_m)v= [\mathcal{P}_m(v\circ\chi_m^{-1})]\circ\chi_m.
\end{equation*}
For a function, $O(t^\infty)$ means that its full Taylor jet vanishes along $t=0$. For an operator, this is understood coefficientwise.
\end{lemma}
The proof, including the normalization, symbol recursion, and passage from coefficient jets to solution jets, is given in Appendix \ref{app:magnetic-boundary-jets}.

\subsection{Boundary-layer solutions}

We use the large parameter $\tau\geq1$ and consider $\tau\to\infty$. We use \textit{boundary-layer quasimodes}: approximate solutions that oscillate tangentially and decay exponentially into the interior on a scale $\tau^{-1}$. The local construction in \cite[Theorem 5.1]{oksanen2024inverse} gives a general quasimode for principal-type operators, which has the following specialization for second-order elliptic operators. Here $p_2(x,\zeta)$ is the homogeneous quadratic principal symbol of $P$, and $O(\tau^{-\infty})$ means $O(\tau^{-N})$ for every $N>0$ as $\tau\to\infty$, in the stated norm. Notice that $ \tau $ corresponds to $h^{-1}$ in the original convention of the cited construction, while in this article $h$ continues to denote boundary data.

\begin{proposition}\label{Prop-boundary-quasimodes}
Let $M$ be a compact smooth manifold with boundary, and let $P$ be a smooth real second-order uniformly elliptic operator on $M$. Near a boundary point choose a boundary coordinate neighborhood $U\subset\partial M$, extend its tangential coordinates $x'$ to inward coordinates $(x',t)$, and fix a nonzero real tangential covector $\xi'$. After shrinking $U$ if necessary, for every smooth cutoff $\chi\in C_c^\infty(U)$ there are quasimodes
\begin{equation*}
 u_\tau=e^{i\tau\Phi}a_\tau,\qquad
 a_\tau\sim\sum_{j\geq0}\tau^{-j}a_j,\qquad
 u_\tau|_{t=0}=e^{i\tau x'\cdot\xi'}\chi(x'),\quad \tau\geq1,
\end{equation*}
supported in a sufficiently small collar over $U$. The phase satisfies
\begin{equation}\label{eq:phasePhi}
 \Phi(x',0)=x'\cdot\xi',\qquad
 \partial_t\Phi(x',0)=z(x'),\qquad
 p_2(x,d\Phi)=0\ \text{to infinite order at }t=0,
\end{equation}
where $z(x')$ is the simple root in the upper half-plane of $p_2(x',0,\xi',z)=0$. Moreover, $\operatorname{Im}\Phi\geq c t$ on the support, $a_0(x',0)=\chi(x')$, and for every fixed $k,\ell\geq0$,
\begin{equation*}
 \|u_\tau\|_{H^k(M)}=O(\tau^{k-1/2}),\qquad
 \|Pu_\tau\|_{H^\ell(M)}=O(\tau^{-\infty}).
\end{equation*}
If $\chi\not\equiv0$, then $\|u_\tau\|_{L^2(M)}\asymp \tau^{-1/2}$.
\end{proposition}
If the Dirichlet problem for $P$ is invertible, then solving $Pr_\tau=-Pu_\tau$, $r_\tau|_{\partial M}=0$, gives exact solutions $u_\tau+r_\tau$ with $r_\tau=O(\tau^{-\infty})$ in every fixed Sobolev norm. When the principal symbol is a nonzero real multiple of that of a metric $g$, its upper root in inward boundary normal coordinates gives
\begin{equation*}
 d\Phi|_{t=0}=\xi'+i|\xi'|_g\nu.
\end{equation*}
Here $\nu=dt$ is the inward unit conormal. By contrast, the form-valued coefficients $\nu_i$ in \eqref{def:normal-derivative} use the outward boundary orientation, so $\nu_i\mathcal{J}^i$ is the outward boundary current form.

\subsection{Multiple linearization}
At an admissible background in Definition \ref{def:admissible-background}, we perform \textit{multiple linearization} of \eqref{eq:quasi_gen}. Let $K\in\mathbb{Z}_+$, choose $(h_1,\ldots,h_K)\in(C^\infty(\partial M))^K$, and set
\begin{equation*}
 f_\epsilon=f_0+\sum_{l=1}^K\epsilon^l h_l,
 \qquad \epsilon=(\epsilon^1,\ldots,\epsilon^K)\in\mathbb{R}^K,
\end{equation*}
for $\epsilon$ close to zero. Lemma \ref{lmm:forward} justifies repeated differentiation of the equation, the solution map, and the DN map with respect to $\epsilon$. In the proof of Theorem \ref{thm:main} we take $f_0\equiv s$ to be constant, so that $u^{f_0}\equiv s$ and $du^{f_0}=0$. The formulae below remain valid for a general background solution.

For a nonempty set $B=\{b_1,\ldots,b_r\}\subset\{1,\ldots,K\}$, let $ h_B:=\{h_{b_1},\ldots ,h_{b_r}\} $ and write
\begin{equation}\label{def-lin-sol}
 w^{h_B}_{f_0}:=
 \left.\partial_{\epsilon^{b_1}}\cdots\partial_{\epsilon^{b_r}}
 u^{f_\epsilon}\right|_{\epsilon=0}.
\end{equation}
The first linearized equation is
\begin{equation}\label{eq:1st-order}
 \left\{
 \begin{aligned}
 \partial_i\left(\partial_{\xi_j}\mathcal{J}^i
 \partial_jw^{h_1}+\partial_s\mathcal{J}^i w^{h_1}\right)&=0
 &&\text{in }M,\\
 w^{h_1}&=h_1&&\text{on }\partial M,
 \end{aligned}\right.
\end{equation}
where coefficients are evaluated at $(x,u^{f_0},du^{f_0})$.

For convenience of higher-order calculations, we use the notation:
\begin{equation}\label{def-QV}
 Q^\alpha\mathcal{V}:=
 \begin{cases}
 \partial_s\mathcal{V},&\alpha=0,\\
 \partial_{\xi_l}\mathcal{V},&\alpha=l\in\{1,\ldots,n\},
 \end{cases}
 \qquad
 P_\alpha u:=
 \begin{cases}
 u,&\alpha=0,\\
 \partial_l u,&\alpha=l\in\{1,\ldots,n\}.
 \end{cases}
\end{equation}
An index $\alpha=0$ in $Q^\alpha$ denotes differentiation with respect to the potential variable $s$, whereas $\alpha=j\geq1$ denotes differentiation with respect to the gradient variable $\xi_j$. These coefficient derivatives are distinct from the spatial derivatives $\partial_j$ acting on solutions in $P_j$. For $K\geq1$, let
\begin{equation}\label{def-QKV}
 (Q_K\mathcal{V})^{i,\alpha_1\cdots\alpha_K}
 :=Q^{\alpha_K}\cdots Q^{\alpha_1}\mathcal{V}^i,
\end{equation}
and thus we can rewrite the first-order linearized equation \eqref{eq:1st-order} as
\begin{equation}\label{eq:1st-order-re}
 \partial_i\left((Q_1\mathcal{J})^{i,\alpha}
 (x,u^{f_0},du^{f_0})P_\alpha w^{h_1}\right)=0.
\end{equation}
Moreover, we denote
\begin{equation*}
 T^{i,K}_{f_0}(v_1,\ldots,v_K):=
 (Q_K\mathcal{J})^{i,\alpha_1\cdots\alpha_K}
 (x,u^{f_0},du^{f_0})
 \prod_{r=1}^K P_{\alpha_r}v_r.
\end{equation*}
This operator is symmetric in its arguments.

Let $\Pi(K)$ be the set of set partitions of $\{1,\ldots,K\}$, namely
\begin{equation*}
    \Pi(K)=\left\{\pi=\{B_1,\ldots,B_{|\pi|}\}:\bigsqcup\limits_{r=1}^{|\pi|} B_r=\{1,\ldots,K\}\right\}.
\end{equation*}
The chain and product rules give the general $K$-th linearized equation
\begin{equation}\label{eq:higher-order-gen}
 \partial_i\left(
 \sum_{\pi=\{B_1,\ldots,B_{|\pi|}\}\in\Pi(K)}
 T^{i,|\pi|}_{f_0}
 (w^{h_{B_1}}_{f_0},\ldots,w^{h_{B_{|\pi|}}}_{f_0})
 \right)=0.
\end{equation}
For $K\geq2$, $w^{h_{\{1,\ldots,K\}}}_{f_0}=0$ on $\partial M$. The corresponding DN derivative is
\begin{equation}\label{eq:higher-order-DN}
 (D^KN_{\mathcal{J}})_{f_0}(h_1,\ldots,h_K)
 =
 \sum_{\pi=\{B_1,\ldots,B_{|\pi|}\}\in\Pi(K)}
 T^{i,|\pi|}_{f_0}
 (w^{h_{B_1}}_{f_0},\ldots,w^{h_{B_{|\pi|}}}_{f_0})\nu_i.
\end{equation}

\subsection{A tensor symmetry lemma}
The following simple lemma is used in the higher-order linearization arguments.
\begin{lemma}\label{lmm:LAsym}
If a tensor is symmetric under interchange of two indices and antisymmetric under interchange of another pair with at least one index position in common, then it vanishes.
\end{lemma}
\begin{proof}
The assertion is immediate if the pairs coincide. Otherwise fix the remaining indices and relabel the three index positions so that the two pairs are $(1,2)$ and $(1,3)$. Successive use of the two symmetries gives
\begin{equation*}
 T_{abc}=T_{bac}=-T_{cab}=-T_{acb}
 =T_{bca}=T_{cba}=-T_{abc}.
\end{equation*}
Thus every component is zero.
\end{proof}

\section{Recovery of the first-order derivatives}\label{sec:first-order}

We prove Theorem \ref{thm:first-order}, assuming only equality of the first linearized measurements stated there. Fix a constant background $f_0\equiv s$, with $u^{f_0}\equiv s$ under the normalized constant-solution convention. Unless stated otherwise, the coefficients in this section are evaluated at $(x,s,0)$. Condition \eqref{condition:regularity} makes both $\partial_\xi\mathcal{J}(x,s,0)$ and $\partial_s\mathcal{J}(x,s,0)$ analytic in $x$, so the induced linear operators have analytic spatial coefficients at this fixed background.

For $m=1,2$, let $w_m^h$ solve \eqref{eq:1st-order-re} for $\mathcal{J}_m$. In Section \ref{sec:first-order} and Section \ref{sec:higher-order}, equality of measured DN maps and their derivatives is understood on $\Gamma$ and on directions in $\mathcal{H}_\Gamma$. The proof proceeds through magnetic reduction, boundary determination, global analytic recovery, and comparison of the original current coefficients.

\subsection{Reduction to a magnetic Schr\"odinger equation}
\label{subsec:magnetic-reduction}

For one current map, retain the notation
\begin{equation}\label{def:1st-order-notations}
 \gamma^{ij}:=\frac{1}{2}\left((Q_1\mathcal{J})^{i,j}
 +(Q_1\mathcal{J})^{j,i}\right),\qquad
 \varphi^{ij}:=(Q_1\mathcal{J})^{i,j}-\gamma^{ij},\qquad
 \Xi^i:=(Q_1\mathcal{J})^{i,0}.
\end{equation}
Then $\gamma$ is symmetric positive definite and $\varphi$ is antisymmetric, and the first linearized equation is
\begin{equation}\label{eq:1st-order-r}
 \partial_i(\gamma^{ij}\partial_jw)
 +(\partial_i\varphi^{ij}+\Xi^j)\partial_jw
 +(\partial_j\Xi^j)w=0.
\end{equation}
Since $n\geq3$, $\gamma$ determines a unique metric $g$ by
\begin{equation}\label{eq:gamma-metric}
 \gamma^{ij}=:\sqrt{\det g}\,g^{ij},\qquad
 \det g=(\det\gamma)^{2/(n-2)}.
\end{equation}
Define the real two-form $F$ and the real one-form $X$ by
\begin{equation}\label{eq:F-X}
 \varphi^{ij}:=\sqrt{\det g}\,g^{ik}F_{kl}g^{lj},
 \qquad
 \Xi^j:=\sqrt{\det g}\,g^{jk}X_k.
\end{equation}
With the codifferential convention of Subsection \ref{subsec:magnetic-tools}, the components are
\begin{equation*}
 (d_g^\star F)^j
 =-\frac{1}{\sqrt{\det g}}\partial_i
 \left(\sqrt{\det g}\,F^{ij}\right),\qquad
 d_g^\star X
 =-\frac{1}{\sqrt{\det g}}\partial_i
 \left(\sqrt{\det g}\,X^i\right).
\end{equation*}
A direct expansion of \eqref{eq:1st-order-r} therefore gives
\begin{equation}\label{eq:1st-order-g}
 -\frac{1}{\sqrt{\det g}}Lw
 =-\Delta_gw-\langle X-d_g^\star F,dw\rangle_g
 +(d_g^\star X)w,
\end{equation}
where $Lw$ denotes the left-hand side of the divergence-form equation.

Set
\begin{equation}\label{eq:def-A}
 A:=-\frac{i}{2}(X-d_g^\star F),\qquad
 a:=iA=\frac{1}{2}(X-d_g^\star F).
\end{equation}
Thus $a$ is real and $A$ is purely imaginary, with $g^{jk}A_jA_k=-|a|_g^2$. The expansion in Subsection \ref{subsec:magnetic-tools} turns \eqref{eq:1st-order-g} into
\begin{equation}\label{eq:1st-order-m}
 \mathcal{P}_{g,A,q} w:=(d_{\overline{A}}^\star d_A+q)w=0,
 \qquad
 q:=d_g^\star(iA)-g^{jk}A_jA_k
 =\frac{1}{2}d_g^\star X+\frac{1}{4}|X-d_g^\star F|_g^2.
\end{equation}

Write $\mathcal{P}=\mathcal{P}_{g,A,q}$ for this induced operator and retain the magnetic Schr\"odinger map \eqref{def:lin-Sch-DN}. Translation covariance gives the following structural property.

\begin{lemma}\label{lmm:coclosed-magnetic-potential}
Let $\mathcal{J}$ satisfy \eqref{condition:elliptic}, including translation covariance, and let $f_0$ be an admissible background in Definition \ref{def:admissible-background}. Define $g,F,X,A,q$ by \eqref{def:1st-order-notations}--\eqref{eq:1st-order-m}, evaluating the current derivatives at $(x,u^{f_0},du^{f_0})$. Then $L1=0$ and
\begin{equation*}
 d_g^\star(iA)=0.
\end{equation*}
Consequently, $q=-g^{jk}A_jA_k=|iA|_g^2$, and, with $a=iA$,
\begin{equation}\label{eq:drift-form}
 \mathcal{P}=-\Delta_g-2\langle a,d\,\cdot\,\rangle_g .
\end{equation}
For a background of class $C^{2,\alpha}$, the differentiated identities are understood distributionally. At constant backgrounds they hold pointwise with analytic coefficients.
\end{lemma}
\begin{proof}
Fix $f_0$ and set $f_\epsilon=f_0+\epsilon$. Translation covariance in \eqref{condition:elliptic} gives $u^{f_\epsilon}=u^{f_0}+\epsilon$, so differentiation at $\epsilon=0$ yields the first linearized solution $w\equiv1$. Hence
\begin{equation*}
 L1=\partial_j\Xi^j=0.
\end{equation*}
By \eqref{eq:F-X},
\begin{equation*}
 d_g^\star X
 =-\frac{1}{\sqrt{\det g}}\partial_j
 \left(\sqrt{\det g}\,X^j\right)
 =-\frac{1}{\sqrt{\det g}}\partial_j\Xi^j=0.
\end{equation*}
Since $(d_g^\star)^2=0$ and $iA=\frac{1}{2}(X-d_g^\star F)$ by \eqref{eq:def-A},
\begin{equation*}
 d_g^\star(iA)
 =\frac{1}{2}\left(d_g^\star X-(d_g^\star)^2F\right)=0.
\end{equation*}
The formula for $q$ follows from \eqref{eq:1st-order-m}. Substituting these two identities in that equation gives \eqref{eq:drift-form}.
\end{proof}
The following lemma shows that the first linearized DN map and the magnetic Schr\"odinger DN map differ by a boundary differential operator.
\begin{lemma}\label{lmm:DN-diff}
Let $\mathcal{J}$ and $f_0$ satisfy the hypotheses of Lemma \ref{lmm:forward}, and define $\gamma,\varphi,\Xi,g,A,q$ by \eqref{def:1st-order-notations}--\eqref{eq:1st-order-m} at $(x,u^{f_0},du^{f_0})$. For every $h\in C^\infty(\partial M)$, the original first linearized DN map and the magnetic Schr\"odinger map \eqref{def:lin-Sch-DN} differ by the boundary differential operator
\begin{equation}\label{eq:DN-difference}
 (D^1N_{\mathcal{J}})_{f_0}(h)-\Lambda_{g,A,q}(h)
 =
 \nu_i\left(
 \varphi^{ij}\partial_jh
 +\frac{1}{2}(\Xi^i-\partial_k\varphi^{ki})h
 \right).
\end{equation}
Here $\nu_i$ has exactly the form-valued convention of \eqref{def:normal-derivative}. No choice of a unit conormal is made in this formula. The contraction $\nu_i\varphi^{ij}\partial_jh$ is tangential and does not depend on how $h$ is extended off the boundary.
\end{lemma}
\begin{proof}
By \eqref{eq:def-A} and \eqref{eq:F-X}, one has
\begin{equation*}
    \gamma^{ij}(iA_j)
 =\frac{1}{2}\sqrt{\det g}\,g^{ij}
 \left(X_j-(d_g^\star F)_j\right)
 =\frac{1}{2}\left(\Xi^i+\partial_k\varphi^{ki}\right),
\end{equation*}
where $\sqrt{\det g}\,(d_g^\star F)^i=-\partial_k\varphi^{ki}$. Consequently, subtracting the density-valued magnetic conormal derivative from the original boundary current form gives
\begin{align*}
\nu_i\left((\gamma^{ij}+\varphi^{ij})\partial_jw+\Xi^iw\right)
-\nu_i\sqrt{\det g}g^{ij} (\partial_j+iA_j)w=
\nu_i\left(
\varphi^{ij}\partial_jw
+\frac{1}{2}(\Xi^i-\partial_k\varphi^{ki})w\right).
\end{align*}
Restricting $w$ to its boundary value $h$ proves \eqref{eq:DN-difference}.
\end{proof}


The next lemma determines the correction in Lemma \ref{lmm:DN-diff} from the first linearized DN map, reducing the original problem to recovering $Q_1\mathcal{J}$ from the induced magnetic Schr\"odinger DN map.

\begin{lemma}\label{lmm:bdr-1}
Let the two currents on the manifold of Theorem \ref{thm:main} satisfy \eqref{condition:regularity}, \eqref{condition:elliptic}, and \eqref{condition:constant-solutions}. Fix a common admissible constant background $s$ in Definition \ref{def:admissible-background}, and define the induced coefficients by \eqref{def:1st-order-notations}--\eqref{eq:1st-order-m} at $(x,s,0)$. If $(D^1N_{\mathcal{J}_1})_sh=(D^1N_{\mathcal{J}_2})_sh$ on $\Gamma$ for every $h\in\mathcal{H}_\Gamma$, then the boundary metric, $\nu_i\varphi_m^{ij}$, and $\nu_i(\Xi_m^i-\partial_k\varphi_m^{ki})$ agree for $m=1,2$ on $\Gamma$ under the canonical boundary identification. Consequently,
\begin{equation*}
 \Lambda_{g_1,A_1,q_1}h=\Lambda_{g_2,A_2,q_2}h\quad\text{on }\Gamma,\qquad h\in\mathcal{H}_\Gamma.
\end{equation*}
\end{lemma}
The result can follow from Lemma \ref{lmm:DN-diff} and the principal-symbol calculation in \cite[Lemma 8.6]{dossantosferreira2009limiting}, but we give the calculation below for completeness of the article.
\begin{proof}[Proof of Lemma \ref{lmm:bdr-1}]
Fix the same tangential coordinates $x'$ on a boundary coordinate neighborhood $U\subset\Gamma$. For each $m=1,2$, independently construct the inward $g_m$-boundary normal coordinates $(x',t_m)$. Thus
\begin{equation*}
 g_m=dt_m^2+h_m(x',t_m),\qquad H_m:=\det h_m.
\end{equation*}
All symbol comparisons below are made after pullback to the common tangential coordinates $x'$.

Let $d\sigma_0$ be the positive coordinate density on the boundary, and define the scalar components $V_m^a$ by $\nu_i\varphi_m^{ia}=V_m^a\,d\sigma_0$. The density-valued principal symbol is
\begin{equation}\label{eq:flux-principal-symbol}
 \sigma_1(D^1N_{\mathcal{J}_m})(x',\xi')
 =\left(\sqrt{H_m}\,|\xi'|_{h_m}+iV_m^a\xi_a\right)d\sigma_0.
\end{equation}
We include a boundary-layer calculation. Take a real cutoff $\chi$ in the common tangential chart and let $u_\tau=\chi e^{i\tau x'\cdot\xi'}$ on the boundary. Proposition \ref{Prop-boundary-quasimodes}, applied to the magnetic operator, gives exact solutions
\begin{equation*}
 w_{m,\tau}=e^{i\tau\Phi_m}a_{m,\tau}+r_{m,\tau},\qquad
 a_{m,\tau}|_{t_m=0}=\chi,\qquad
 d\Phi_m|_{t_m=0}=\xi'+i|\xi'|_{h_m}dt_m,
\end{equation*}
where $r_{m,\tau}|_{\partial M}=0$ and $r_{m,\tau}=O(\tau^{-\infty})$ in every fixed Sobolev norm. Since $t_m$ is inward, the symmetric principal part of the boundary current form is $-\sqrt{H_m}\,\partial_{t_m}w_{m,\tau}\,d\sigma_0$. The normal phase derivative therefore contributes $\tau\sqrt{H_m}|\xi'|_{h_m}u_\tau\,d\sigma_0$. The contribution of the antisymmetric current coefficient involves only tangential derivatives and contributes $i\tau V_m^a\xi_a u_\tau\,d\sigma_0$. Derivatives of amplitudes and the $\Xi_m$ term contribute only $O(1)$ to the pairing. Thus
\begin{equation*}
 \int_{\partial M}\overline{u_\tau}\,
 (D^1N_{\mathcal{J}_m})_{f_0}(u_\tau)
 =\tau\int_{\partial M}\chi^2
 \left(\sqrt{H_m}|\xi'|_{h_m}+iV_m^a\xi_a\right)d\sigma_0+O(1).
\end{equation*}
Localizing $\chi$ after multiplication by $\tau^{-1}$ and passage to the limit $\tau\to\infty$ gives \eqref{eq:flux-principal-symbol}. Writing the scalar coefficient of the symbol as $S_m(\xi')$, we obtain
\begin{equation*}
 \frac{S_m(\xi')+S_m(-\xi')}{2}=\sqrt{H_m}|\xi'|_{h_m},\qquad
 \frac{S_m(\xi')-S_m(-\xi')}{2i}=V_m^a\xi_a.
\end{equation*}
Squaring the even part determines
\begin{equation*}
 K_m^{ab}:=H_mh_m^{ab}.
\end{equation*}
Since the boundary dimension is $n-1$ and $n\geq3$, we know that $ \det K_m=H_m^{n-2} $, or equivalently $ H_m=(\det K_m)^{\frac{1}{n-2}} $.
Consequently, we can determine the boundary metric by
\begin{equation*}
 h_m^{ab}=(\det K_m)^{-\frac{1}{n-2}}K_m^{ab}.
\end{equation*}

Meanwhile, the odd part of \eqref{eq:flux-principal-symbol} determines $V_m$. Lemma \ref{lmm:coclosed-magnetic-potential} gives $w_m^1=1$, and hence $(D^1N_{\mathcal{J}_m})_s(1)|_\Gamma=\partial_sN_{\mathcal{J}_m}^\Gamma(s,0)=\nu_i\Xi_m^i|_\Gamma$. In boundary normal coordinates, antisymmetry gives $\varphi_m^{nn}=0$, so $\partial_k\varphi_m^{kn}=\partial_a\varphi_m^{an}$ on the boundary. The components of the normal density are constant in these coordinates, and hence
\begin{equation*}
 \nu_i\partial_k\varphi_m^{ki}=-(\partial_aV_m^a)d\sigma_0.
\end{equation*}
It follows that the zeroth-order coefficient in the correction operator is
\begin{equation*}
 \frac{1}{2}\nu_i(\Xi_m^i-\partial_k\varphi_m^{ki})
 =\frac{1}{2}\left(\nu_i\Xi_m^i+(\partial_aV_m^a)d\sigma_0\right),
\end{equation*}
whose two terms have already been recovered from the constant datum and the odd part of the principal symbol, respectively. The already recovered tangential field $V_m$, including its tangential derivatives, therefore determines the whole correction operator in \eqref{eq:DN-difference}. Equality of the original first linearized DN maps implies equality of these correction operators and then equality of the magnetic Schr\"odinger DN maps.
\end{proof}

\subsection{Analytic globalization from the boundary jets}
\label{subsec:analytic-extension}

We prove Theorem \ref{thm:magnetic-recovery}. We first remove scalar conjugation using the constant solution. Solution coordinates then turn the formal boundary identification into an analytic collar map, which extends to the interior via Poisson embeddings. Finally, coclosedness of the drift and Cauchy uniqueness at $\Gamma$ remove the remaining conformal factor.

Write $\mathcal{P}_m=\mathcal{P}_{g_m,A_m,q_m}$. The assumptions $A_m=-ia_m$, $d_{g_m}^\star a_m=0$, and $q_m=|a_m|_{g_m}^2$ give
\begin{equation*}
 \mathcal{P}_m=-\Delta_{g_m}-2\langle a_m,d\,\cdot\,\rangle_{g_m}.
\end{equation*}
The maximum principle and classical Dirichlet theory make the zero-Dirichlet realization invertible, as in the proof of Lemma \ref{lmm:forward}. Let $w_m^h$ be its Dirichlet solution with boundary value $h$. In particular, $w_m^1\equiv1$.

The common magnetic DN map also determines the boundary metric on $\Gamma$. Indeed, in boundary coordinates with coordinate density $d\sigma_0$, its principal symbol is $\sqrt{\det b_m}\,|\xi'|_{b_m}\,d\sigma_0$, where $b_m$ is the induced boundary metric. Squaring the scalar symbol determines $K_m=(\det b_m)b_m^{-1}$. Since $\det K_m=(\det b_m)^{n-2}$ and $n\geq3$, we obtain $b_1=b_2$. Lemma \ref{lmm:magnetic-boundary-jets} therefore applies. Retain its notation $\chi_m,\mathcal{C},\kappa_0,\lambda_0,U_m^h$ for the auxiliary smooth collars and gauges.

Because $U_m^1\equiv1$, taking $h=1$ in \eqref{eq:formal-solution-package} gives $\kappa_0-1=O(t^\infty)$. All derivatives of $\kappa_0-1$ vanish at the boundary, so the product rule shows that conjugation by $\kappa_0$ leaves every coefficient jet unchanged. Therefore the boundary identities \eqref{eq:formal-solution-package} reduce to
\begin{equation}\label{eq:unital-formal-package}
 \mathcal{P}_2^{\mathcal{C}}
 -\lambda_0\mathcal{P}_1^{\mathcal{C}}=O(t^\infty),
 \qquad U_2^h-U_1^h=O(t^\infty).
\end{equation}
Here the constant solution removes the scalar conjugation from the boundary jets. The passage to an actual collar identification follows the $Z$-coordinate argument of \cite[Propositions~3.2--3.5 and Lemma~3.6]{LassasMatti2022TCpf} and is stated in the following lemma. That argument concerns the conformal Laplacian, while we give its operator version below, including the real drift and the correspondence of Dirichlet solutions with data in $\mathcal{H}_\Gamma$.

\begin{lemma}[From formal boundary jets to an analytic collar]
\label{lmm:formal-to-collar}
Under the hypotheses of Theorem \ref{thm:magnetic-recovery}, let $\mathcal{P}_m=\mathcal{P}_{g_m,A_m,q_m}$. For every $p\in\Gamma$ there are neighborhoods $V_m$ of $p$ in $M$, an orientation-preserving boundary-fixing real analytic diffeomorphism $\Phi_{\mathrm{col}}:V_2\to V_1$, and a positive real analytic function $\lambda$ on $V_2$ such that
\begin{equation}\label{eq:collar-operator-identity}
 \mathcal{P}_2
 =\lambda\,\Phi_{\mathrm{col}}^*\mathcal{P}_1,\qquad
 \lambda>0,\qquad
 \lambda|_{V_2\cap\partial M}=1,\quad d\lambda|_{V_2\cap\partial M}=0,
\end{equation}
and
\begin{equation}\label{eq:collar-diffeomorphism}
 w_2^h=w_1^h\circ\Phi_{\mathrm{col}}
 \quad\text{on }V_2
 \quad\text{for every }h\in\mathcal{H}_\Gamma.
\end{equation}
\end{lemma}
\begin{proof}
Lemma \ref{lmm:magnetic-boundary-jets} allows us to use \eqref{eq:unital-formal-package} and \eqref{eq:formal-boundary-normalization}. Following \cite[Proposition~3.2]{LassasMatti2022TCpf}, we choose analytic coordinates $y=(y^1,\ldots,y^{n-1})$ on a boundary neighborhood $U\Subset\Gamma$ of $p$ and construct coordinates from Dirichlet solutions. The denominator in that construction is $w_m^1\equiv1$ here. The auxiliary normalized collars need not be analytic, so their Taylor identities must first be transferred to these solution coordinates.

Choose a smaller neighborhood $U_0\Subset U$ of $p$ and boundary data $h_a\in C_c^\infty(\Gamma)$, $1\leq a\leq n-1$, equal to $y^a$ on $U_0$. Choose also $0\leq h_n\in C_c^\infty(\Gamma\setminus\overline{U}_0)$, not identically zero. For these real drift operators, the strong maximum principle and the Hopf boundary lemma \cite[Theorem~3.5 and Lemma~3.4]{gilbarg2001elliptic} give
\begin{equation*}
 w_m^{h_n}>0\quad\text{in }M^\circ,\qquad
 \partial_{\nu_{\mathrm{in}}}w_m^{h_n}>0
 \quad\text{at the chosen boundary point}.
\end{equation*}
Lemma \ref{lmm:analytic-boundary} applies because all the boundary data just chosen are analytic near the point. It follows that
\begin{equation*}
 Z_m=(w_m^{h_1},\ldots,w_m^{h_n})
\end{equation*}
is a real analytic coordinate map up to the boundary. The common tangential block and the two positive inward derivatives show that $\operatorname{sgn}\det dZ_1=\operatorname{sgn}\det dZ_2$. Hence $Z_1^{-1}\circ Z_2$ preserves orientation.

For a diffeomorphism $\Psi:U\to V$ and a differential operator $P$ on $U$, we use the operator push-forward
\begin{equation}\label{eq:operator-pushforward-definition}
 (\Psi_*P)v:=((\Psi^{-1})^* P)v=\bigl[P(v\circ\Psi)\bigr]\circ\Psi^{-1},
 \qquad v\in C^\infty(V).
\end{equation}
Thus $\Psi_*P$ is the expression of $P$ in the $\Psi$ coordinates.

Recall the boundary-fixing collar map $ \chi_m:\mathcal{C}\to M $ defined in Lemma \ref{lmm:magnetic-boundary-jets}. Put $\widehat{Z}_m :=Z_m\circ\chi_m$. These expressions on the auxiliary collar are only smooth, but the operator
\begin{equation*}
    \mathcal{P}_{m,Z}:=(\widehat{Z}_m)_*(\mathcal{P}_m^{\mathcal{C}})=(Z_m\circ\chi_m)_*(\chi_m^*\mathcal{P}_m)=(Z_m)_*\mathcal{P}_m
\end{equation*}
is analytic, since both $Z_m$ and $\mathcal{P}_m$ are analytic up to the boundary. Write $z=(z_1,\ldots,z_n)$ for the solution coordinates, in which the common boundary is $z_n=0$. Analyticity in these coordinates allows the formal comparison to become exact.

The second relation in \eqref{eq:unital-formal-package} says $\widehat{Z}_1-\widehat{Z}_2=O(t^\infty)$. In particular, the inverse jets of these two coordinate maps agree. As in the jet-transfer argument of \cite[Proposition~3.5]{LassasMatti2022TCpf}, a second-order operator transformation uses only finitely many derivatives of the coordinate map and its inverse. The equality of the coordinate jets therefore gives
\begin{equation*}
 (\widehat{Z}_2)_*\mathcal{P}_1^{\mathcal{C}}
 -(\widehat{Z}_1)_*\mathcal{P}_1^{\mathcal{C}}=O(z_n^\infty).
\end{equation*}
Define $\lambda_{0,Z}:=(\widehat{Z}_2)_*\lambda_0=\lambda_0\circ\widehat{Z}_2^{-1}$, where the push-forward of a scalar function is its composition with the inverse map. Pushing forward the first relation in \eqref{eq:unital-formal-package} by $\widehat{Z}_2$ and using the preceding identity gives
\begin{equation}\label{eq:operator-conformal}
\mathcal{P}_{2,Z}-\lambda_{0,Z}\mathcal{P}_{1,Z}=(\widehat{Z}_2)_*(\mathcal{P}_2^{\mathcal{C}}-\lambda_0\mathcal{P}_1^{\mathcal{C}})+\lambda_{0,Z}\left((\widehat{Z}_2)_*\mathcal{P}_1^{\mathcal{C}}-(\widehat{Z}_1)_*\mathcal{P}_1^{\mathcal{C}}\right)=O(z_n^\infty).
\end{equation}
Let $G_m$ be the principal coefficient matrix of $\mathcal{P}_{m,Z}$ and define the analytic function
\begin{equation*}
 \lambda:=\frac{1}{n}\operatorname{tr}(G_2G_1^{-1}).
\end{equation*}
Then $\lambda-\lambda_{0,Z}=O(z_n^\infty)$ by \eqref{eq:operator-conformal}, while every coefficient of $\mathcal{P}_{2,Z}-\lambda\mathcal{P}_{1,Z}$ is analytic and has vanishing full boundary jet. Each coefficient therefore has a convergent power series equal to zero at the boundary, and hence, by analytic continuation in a collar neighborhood, we obtain the exact operator identity
\begin{equation*}
    \mathcal{P}_{2,Z}-\lambda\mathcal{P}_{1,Z}\equiv 0.
\end{equation*}
Define the boundary-fixing analytic diffeomorphism
\begin{equation*}
 \Phi_{\mathrm{col}}:=Z_1^{-1}\circ Z_2.
\end{equation*}
Applying $Z_2^*$ to the operator identity gives \eqref{eq:collar-operator-identity}, with $\lambda\circ Z_2$ denoted again by $\lambda$. Its boundary values $\lambda=1$ and $d\lambda=0$ follow from the corresponding jets in \eqref{eq:formal-boundary-normalization}. In particular, $\lambda>0$ after shrinking the collar.

To identify the Dirichlet solutions with data in $\mathcal{H}_\Gamma$, we match their Cauchy data as in the local Poisson comparison of \cite[Lemma~3.1]{lassas2020poisson}. The magnetic conormal calculation for our operator is as follows. Put $\widehat{g}_1=\Phi_{\mathrm{col}}^* g_1$ and $\widehat{a}_1=\Phi_{\mathrm{col}}^*(iA_1)$. Comparing the principal and first-order coefficients in \eqref{eq:collar-operator-identity} gives
\begin{equation*}
 g_2=\lambda^{-1}\widehat{g}_1,\qquad
 iA_2=\widehat{a}_1+d\omega,\qquad
 \omega=\frac{n-2}{4}\log\lambda .
\end{equation*}
Let $\widehat{u}=w_1^h\circ\Phi_{\mathrm{col}}$, and denote the density-valued magnetic conormal trace by $\mathcal{N}_{\mathcal{P}}$. The conformal relations give
\begin{equation*}
 \mathcal{N}_{\mathcal{P}_2}(\widehat{u})
 =\iota^*\!\left[
 \lambda^{1-n/2}
 (d\widehat{u}+\widehat{a}_1\widehat{u}+d\omega\,\widehat{u})
 ^{\sharp_{\widehat{g}_1}}\lrcorner d\operatorname{Vol}_{\widehat{g}_1}
 \right].
\end{equation*}
Since $\lambda=1$ and $d\lambda=0$ on the boundary, and $\Phi_{\mathrm{col}}|_{V_2\cap\partial M}=\operatorname{Id}$, its boundary value equals the pullback of $\mathcal{N}_{\mathcal{P}_1}(w_1^h)$. Equality of the magnetic Schr\"odinger maps therefore gives equal Cauchy data for $w_2^h$ and $\widehat{u}$. Both solve $\mathcal{P}_2v=0$ by \eqref{eq:collar-operator-identity}. The analytic boundary is noncharacteristic for this elliptic operator, so Holmgren uniqueness proves $ w_2^h = \widehat{u} = w_1^h\circ\Phi_{\mathrm{col}} $, namely \eqref{eq:collar-diffeomorphism}. This argument applies to all boundary data in $\mathcal{H}_\Gamma$. Only the finitely many coordinate-generating traces needed analytic boundary regularity.
\end{proof}

The operators $\mathcal{P}_m$ satisfy the hypotheses of Lemma \ref{lmm:compactified-poisson}, so we can apply the embedding and global extension results from Subsection \ref{subsec:poisson-tools}.

\begin{proof}[Proof of Theorem \ref{thm:magnetic-recovery}]
Lemma \ref{lmm:formal-to-collar} supplies orientation-preserving local identifications that preserve the Dirichlet solutions with data in $\mathcal{H}_\Gamma$. Lemma \ref{lmm:poisson-continuation} extends them to a unique orientation-preserving $\Gamma$-fixing real analytic diffeomorphism $\Phi:M^\circ\cup\Gamma\to M^\circ\cup\Gamma$. This proves the solution identity first for data in $C_c^\infty(\Gamma)$ and then for all $h\in\mathcal{H}_\Gamma$ by linearity and $w_m^1=1$.

It remains to prove that the factor in \eqref{eq:collar-operator-identity} is identically one. The assumed coclosedness of the drift, together with the boundary normalization on $\Gamma$, gives this conclusion through Cauchy uniqueness. Write
\begin{equation*}
 \widehat{\mathcal{P}}_1=\Phi^*\mathcal{P}_1,\qquad
 \widehat{g}_1=\Phi^* g_1,\qquad
 \widehat{a}_1=\Phi^*(iA_1).
\end{equation*}
The two principal symbols define bundle isomorphisms $T^* M\to TM$. Their composition therefore has an intrinsic trace. Writing their matrices in any coordinates as $G_2,\widehat{G}_1$, define the analytic positive function
\begin{equation*}
 c=\frac{1}{n}\operatorname{tr}(G_2\widehat{G}_1^{-1}).
\end{equation*}
It is coordinate independent. Because $\Phi$ agrees with the local maps near $\Gamma$, $c$ equals the corresponding collar factor $\lambda$ there. Thus $c=1$ and $dc=0$ on $\Gamma$. Every coefficient of $\mathcal{P}_2-c\widehat{\mathcal{P}}_1$ vanishes on a nonempty collar neighborhood, so analytic continuation on the connected interior gives
\begin{equation}\label{eq:global-operator-factor}
 \mathcal{P}_2=c\widehat{\mathcal{P}}_1\quad\text{on }M^\circ,\qquad
 c|_\Gamma=1,\qquad dc|_\Gamma=0.
\end{equation}
Comparison of its principal and first-order coefficients gives
\begin{equation}\label{eq:global-conformal-drift}
 g_2=c^{-1}\widehat{g}_1,\qquad
 iA_2=\widehat{a}_1+d\omega,\qquad
 \omega=\frac{n-2}{4}\log c.
\end{equation}
For reference, comparison of the zeroth-order coefficient also gives
\begin{equation*}
 q_2=c\left(
 \Phi^* q_1-c^{\frac{n-2}{4}}
 \Delta_{\widehat{g}_1}(c^{\frac{2-n}{4}})
 \right).
\end{equation*}

The remaining factor is constrained by more than equality of solution spaces: both drift one-forms are coclosed, meaning that their codifferentials vanish. By hypothesis and invariance under pullback, $d_{g_2}^\star(iA_2)=d_{\widehat{g}_1}^\star\widehat{a}_1=0$. For a one-form $V$, the conformal relation in \eqref{eq:global-conformal-drift} gives
\begin{equation*}
 c^{-1}d_{g_2}^\star V
 =d_{\widehat{g}_1}^\star V
  +2\langle d\omega,V\rangle_{\widehat{g}_1}.
\end{equation*}
Apply this to $V=iA_2=\widehat{a}_1+d\omega$. Since $d^\star d=-\Delta$, one obtains
\begin{equation*}
 -\Delta_{\widehat{g}_1}\omega
 +2\langle d\omega,\widehat{a}_1+d\omega\rangle_{\widehat{g}_1}=0,
 \qquad \omega|_\Gamma=0,\qquad d\omega|_\Gamma=0.
\end{equation*}
The substitution $z=e^{-2\omega}$ absorbs the quadratic gradient term. It satisfies the linear equation
\begin{equation*}
 \Delta_{\widehat{g}_1}z
 -2\langle\widehat{a}_1,dz\rangle_{\widehat{g}_1}=0,\qquad
 z|_\Gamma=1,\qquad dz|_\Gamma=0.
\end{equation*}
Holmgren uniqueness on the noncharacteristic analytic boundary $\Gamma$ gives $z=1$ near $\Gamma$. Analytic continuation in $M^\circ$ then gives $z\equiv1$, so $\omega=0$ and, since $n\geq3$, $c=1$. Equations \eqref{eq:global-operator-factor} and \eqref{eq:global-conformal-drift} then give the asserted equality of triples.

The resulting interior isometry extends uniquely to the metric completions, namely the compact manifolds $M$. This extension is analytic up to every boundary component. Indeed, in sufficiently small boundary normal collars, write $E_m(p,t)=\exp_p^{g_m}(tN_m)$, where $N_m$ is the inward unit normal, and let $\pi_m$ be the nearest-boundary projection. Preservation of distance and uniqueness of the short minimizing normal geodesic give $\Phi(E_2(p,t))=E_1(F(p),t)$, where $F=\Phi|_{\partial M}$. For a fixed small $t_0>0$, the formula $F=\pi_1\circ\Phi\circ E_2(\cdot,t_0)$ proves analyticity of $F$, and the collar formula proves analyticity of $\Phi$ at the boundary. The inverse is treated in the same way. The extension fixes $\Gamma$, and the equality of triples holds on all of $M$ by continuity.

Finally, if $\widetilde{\Phi}$ is another $\Gamma$-fixing diffeomorphism identifying the triples, covariance and Dirichlet uniqueness give $w_2^h=w_1^h\circ\widetilde{\Phi}$ for $h\in\mathcal{H}_\Gamma$. The boundary data are unchanged because they are constant on $\partial M\setminus\Gamma$. Injectivity of the local Poisson embedding then gives $\widetilde{\Phi}=\Phi$ in the interior, hence on $M$ by continuity.
\end{proof}

\subsection{Recovery of \texorpdfstring{$Q_1\mathcal{J}$}{Q1 J} after the global identification}
\label{subsec:first-order-recovery}

We first show that the identifications supplied by Theorem \ref{thm:magnetic-recovery} agree across constant backgrounds.

\begin{proposition}[Identification independent of the background]\label{prop:analytic-continuation}
Let the two currents on the manifold of Theorem \ref{thm:main} satisfy \eqref{condition:regularity}, \eqref{condition:elliptic}, and \eqref{condition:constant-solutions}, and suppose every constant background is admissible in Definition \ref{def:admissible-background}. Write $(g_{m,s},A_{m,s},q_{m,s})$ for the induced triples \eqref{eq:gamma-metric}, \eqref{eq:def-A}, and \eqref{eq:1st-order-m}. At each constant background $s$ for which $(D^1N_{\mathcal{J}_1})_sh=(D^1N_{\mathcal{J}_2})_sh$ on $\Gamma$ for every $h\in\mathcal{H}_\Gamma$, there is an orientation-preserving $\Gamma$-fixing real analytic diffeomorphism $\Phi_s:M\to M$ such that
\begin{equation}\label{eq:global-solution-identity}
 w_{2,s}^h=w_{1,s}^h\circ\Phi_s
 \quad\text{for every }h\in\mathcal{H}_\Gamma,
\end{equation}
and
\begin{equation}\label{eq:global-triple-identity}
 g_{2,s}=\Phi_s^* g_{1,s},\qquad
 A_{2,s}=\Phi_s^* A_{1,s},\qquad
 q_{2,s}=\Phi_s^* q_{1,s}.
\end{equation}
For each $m$, the first linearized Dirichlet solution map $h\mapsto w_{m,s}^h$ is independent of $s$. Consequently the maps $\Phi_s$ agree for all backgrounds where the first linearized DN maps coincide. Write $\Phi$ for this common map. 
\end{proposition}
\begin{proof}
At each background where the first linearized DN maps coincide, Lemmas \ref{lmm:coclosed-magnetic-potential} and \ref{lmm:bdr-1} verify the hypotheses of Theorem \ref{thm:magnetic-recovery} for the induced triples. That theorem gives \eqref{eq:global-solution-identity} and \eqref{eq:global-triple-identity}.

Fix $s_0\in\mathbb{R}$ and $h\in C^\infty(\partial M)$. For $s$ sufficiently close to $s_0$ and $|\epsilon|$ sufficiently small, the boundary data $s_0+\epsilon h$ and the shift $s-s_0$ satisfy \eqref{condition:elliptic}. Translation covariance gives, for $m=1,2$,
\begin{equation*}
 u_m^{s+\epsilon h}
 =u_m^{s_0+\epsilon h}+(s-s_0).
\end{equation*}
The local branches agree on overlaps by the uniqueness in Lemma \ref{lmm:forward}, since their distinguished solutions are the same constant function.

Differentiating in $\epsilon$ at zero, using the smooth dependence proved there, yields $w_{m,s}^h=w_{m,s_0}^h$. Thus $s\mapsto w_{m,s}^h$ is locally constant for every $h$. Every constant background is admissible, so connectedness of $\mathbb{R}$ makes it constant globally. No uniform admissibility radius in $s$ is needed.

The local Poisson embeddings are therefore independent of $s$. Denote them by $\mathscr{P}_m^\Gamma$. For any two backgrounds $s,s_0$ at which the first linearized DN maps coincide, \eqref{eq:global-solution-identity} gives
\begin{equation*}
 \mathscr{P}_1^\Gamma\circ\Phi_s
 =\mathscr{P}_2^\Gamma
 =\mathscr{P}_1^\Gamma\circ\Phi_{s_0}.
\end{equation*}
Injectivity of $\mathscr{P}_1^\Gamma$ implies $\Phi_s=\Phi_{s_0}$ in the interior, hence on $M$ by continuity. This argument requires no differentiability assumption on the family $\Phi_s$.
\end{proof}

\begin{remark}[Dependence of the identification on the background]\label{rem:background-identification}
The proof that $\Phi_s$ is independent of $s$ uses translation covariance at perturbed boundary data $s+\epsilon h$. The weaker relation $u^{s+\epsilon}=s+u^\epsilon$ for constant boundary data is insufficient: Appendix \ref{app:background-dependent-gauge} constructs currents with identical nonlinear DN maps and a nonconstant family $\Phi_s$. The transformed current depends nonlinearly on the gradient and lies outside the Sun--Uhlmann setting \cite{sun1997inverse}, where the identity $ \Phi_{s+\epsilon} \equiv \Phi_s $ is proved without assuming full translation covariance.
\end{remark}

We now pass from the recovered $ (g,A,q) $ triple to the original current coefficients. The following comparison is made after transporting $\mathcal{J}_2$ by the fixed map of Proposition \ref{prop:analytic-continuation}. All subscript-$2$ coefficients are recomputed from the transported current. Fixing $\Gamma$ preserves the local nonlinear DN maps and their derivatives.

\begin{lemma}\label{lmm:first-order-comparison}
After the fixed identification in Proposition \ref{prop:analytic-continuation}, replace $\mathcal{J}_2$ by $ H_\Phi \mathcal{J}_2$ and fix a constant background $s$. Set $D_m^{ij}=(Q_1\mathcal{J}_m)^{i,j}(x,s,0)$ and $B_m^i=(Q_1\mathcal{J}_m)^{i,0}(x,s,0)$, and let $L_mw=\partial_i(D_m^{ij}\partial_jw+B_m^iw)$. Assume the transported triples and Dirichlet solution families in \eqref{eq:global-triple-identity} and \eqref{eq:global-solution-identity} coincide, with equal original first linearized DN maps. Then there is an antisymmetric contravariant tensor density $\varphi_{(1)}$ such that
\begin{equation}\label{eq:first-order-invisible}
 D_1^{ij}-D_2^{ij}=\varphi_{(1)}^{ij},\qquad
 B_1^i-B_2^i=\partial_j\varphi_{(1)}^{ij},\qquad
 \nu_i\varphi_{(1)}^{ij}|_\Gamma=0.
\end{equation}
Moreover $L_1=L_2$ as differential expressions.
\end{lemma}
\begin{proof}
Equality $g_1=g_2$ and \eqref{eq:gamma-metric} give $\gamma_1=\gamma_2=:\gamma$. Therefore
\begin{equation*}
 \varphi_{(1)}^{ij}:=D_1^{ij}-D_2^{ij}
 =\varphi_1^{ij}-\varphi_2^{ij}
\end{equation*}
is antisymmetric. The identity used in the proof of Lemma \ref{lmm:DN-diff} can be written
\begin{equation*}
 2\gamma^{ij}(iA_m)_j=B_m^i+\partial_k\varphi_m^{ki}.
\end{equation*}
Since $A_1=A_2$, subtraction gives
\begin{equation}\label{eq:B-from-antisymmetric}
 B_1^i-B_2^i=-\partial_k\varphi_{(1)}^{ki}
 =\partial_k\varphi_{(1)}^{ik},
\end{equation}
which is the second identity in \eqref{eq:first-order-invisible}.

For the boundary condition, use the same global solution $w^h=w_1^h=w_2^h$. Equality of the original linearized boundary current forms and \eqref{eq:B-from-antisymmetric} imply
\begin{equation*}
 0=\nu_i\left(
 \varphi_{(1)}^{ij}\partial_jw^h
 +\partial_k\varphi_{(1)}^{ik}w^h\right).
\end{equation*}
At $p\in\Gamma$, choose $h\in C_c^\infty(\Gamma)$ with $h(p)=0$ and arbitrary tangential differential. Since $w^h|_{\partial M}=h$, its tangential differential at $p$ is $d_{\partial M}h(p)$. The normal derivative does not contribute because the vector $\nu_i\varphi_{(1)}^{ij}$ is tangential by antisymmetry. The displayed identity therefore forces this vector to vanish. Finally, for every smooth $w$,
\begin{align*}
 (L_1-L_2)w
 &=\partial_i\left(
 \varphi_{(1)}^{ij}\partial_jw
 +\partial_k\varphi_{(1)}^{ik}w\right)=0.
\end{align*}
Indeed the Hessian and double-divergence terms vanish by antisymmetry, and the two remaining first-derivative terms cancel after interchanging dummy indices.
\end{proof}

\begin{proof}[Proof of Theorem \ref{thm:first-order}]
Lemma \ref{lmm:bdr-1} converts the first linearized measurements into the magnetic Schr\"odinger data. Theorem \ref{thm:magnetic-recovery} and Proposition \ref{prop:analytic-continuation} then give one $\Gamma$-fixing analytic diffeomorphism $\Phi$, independent of $s$, identifying the induced triples and solution families. Lemma \ref{lmm:first-order-comparison} yields \eqref{eq:1st-order-coeff-a}--\eqref{eq:1st-order-coeff-c} and equality of the linearized differential expressions. The formula $\varphi_{(1)}^{ij}=\partial_{\xi_j}(\mathcal{J}_1-H_\Phi\mathcal{J}_2)^i(x,s,0)$ gives the asserted smoothness in $(x,s)$ and spatial analyticity at each fixed $s$.
\end{proof}

Since $\Phi$ is independent of $s$, the recovered coefficient identities may be differentiated in $s$ along the admissible constant backgrounds. In the remainder of the article we keep the replacement $\mathcal{J}_2\mapsto H_\Phi\mathcal{J}_2$ and write
\begin{equation*}
 w^h:=w_{1,s}^h=w_{2,s}^h,\qquad L_s:=L_{1,s}=L_{2,s},
\end{equation*}
where $w^h$ is independent of $s$, whereas $L_s$ may depend on $s$. Equality of the two differential expressions at each $s$ still allows their displayed divergence-form coefficients to differ by \eqref{eq:first-order-invisible}. In the higher-order induction that follows, the background and every occurrence of $Q_K\mathcal{J}_m$ are understood along the same constant family, at $(x,s,0)$.

\section{Recovery of the higher-order derivatives}
\label{sec:higher-order}

Assume the hypotheses of Theorem \ref{thm:main}. Use the fixed global identification from Theorem \ref{thm:first-order} and fix the constant background $f_0\equiv s$. All coefficients below are evaluated at $(x,s,0)$, and the induction is carried out for every $s\in\mathbb{R}$ simultaneously. For each fixed $s$, let $L=L_s$ denote the common first linearized differential expression and $w^h=w_s^h$ its Dirichlet solutions. We suppress this background index during each induction step. Write
\begin{equation*}
 \delta Q_K \mathcal{J}:=Q_K\mathcal{J}_1-Q_K\mathcal{J}_2,\qquad
 \varphi_{(1)}^{ij}=(\delta Q_1 \mathcal{J})^{i,j},
\end{equation*}
and, for $K\geq1$, set
\begin{equation}\label{eq:def-varphi-K}
 \varphi_{(K)}^{ij}:=\partial_s^{K-1}\varphi_{(1)}^{ij}.
\end{equation}
We prove the coefficient identities
\begin{equation}\label{eq:Kth-order-coeff}
 (\delta Q_K \mathcal{J})^{i,\alpha_1\cdots\alpha_K}
 =
 \begin{cases}
 \partial_j\varphi_{(K)}^{ij},
 &\alpha_1=\cdots=\alpha_K=0,\\
 \varphi_{(K)}^{ij},
 &\text{one $\alpha_r=j\geq1$, all other $\alpha_r=0$},\\
 0,&\text{at least two $\alpha_r$ belong to $\{1,\ldots,n\}$},
 \end{cases}
\end{equation}
together with
\begin{equation}\label{eq:varphi-K-boundary}
 \varphi_{(K)}^{ij}=-\varphi_{(K)}^{ji},\qquad
 \nu_i\varphi_{(K)}^{ij}|_\Gamma=0,
\end{equation}
and equality of the nonlinear solution variations through order $K$. The case $K=1$ is Theorem \ref{thm:first-order}.

\subsection{Simultaneous induction and the integral identity}

Let $K\geq2$. Assume the coefficient identities \eqref{eq:Kth-order-coeff}--\eqref{eq:varphi-K-boundary} and equality of the solution variations through order $K-1$, for every background in the same open interval and in fixed coordinates. The interval quantifier permits differentiation in $s$. For $l<K$ and arbitrary smooth $v_1,\ldots,v_l$, the coefficient identities give
\begin{equation}\label{eq:invisible-product}
 (\delta Q_l \mathcal{J})^{i,\alpha_1\cdots\alpha_l}
 \prod_{r=1}^l P_{\alpha_r}v_r
 =\partial_j\left(\varphi_{(l)}^{ij}\prod_{r=1}^l v_r\right).
\end{equation}
Every current of the displayed form has zero divergence by antisymmetry and zero normal component on $\Gamma$: in boundary-adapted coordinates, $\varphi_{(l)}^{na}|_\Gamma=0$ and $\varphi_{(l)}^{nn}=0$, so the normal component is a tangential derivative of a vanishing boundary trace. We use this cancellation throughout the induction.

For the intermediate partitions $2\leq|\pi|\leq K-1$ in \eqref{eq:higher-order-gen} and \eqref{eq:higher-order-DN}, the arguments are common lower-order variations, so \eqref{eq:invisible-product} cancels their differences in both the equation and its boundary measurement. The single-block partition gives the common first linearized operator applied to the difference of the $K$-th variations.

Put
\begin{equation*}
 z=w_1^{h_{\{1,\ldots,K\}}}-w_2^{h_{\{1,\ldots,K\}}},\qquad
 T_K^i=(\delta Q_K \mathcal{J})^{i,\alpha_1\cdots\alpha_K}
           \prod_{r=1}^K P_{\alpha_r}w^{h_r}.
\end{equation*}
For $h_r\in\mathcal{H}_\Gamma$, the remaining equation and measured boundary identity are
\begin{equation}\label{eq:higher-order-diff}
 Lz+\partial_iT_K^i=0,\qquad z|_{\partial M}=0,\qquad
 \nu_i(D^{ij}\partial_jz+B^iz+T_K^i)|_\Gamma=0.
\end{equation}
Here $Lz=\partial_i(D^{ij}\partial_jz+B^iz)$ may use either current's first-order coefficients. The difference between the two such representations is itself the invisible product at order one.

Let $L^\star$ denote the formal transpose, so that in these coordinates
\begin{equation*}
 L^\star v=\partial_j(D^{ij}\partial_i v)-B^i\partial_i v.
\end{equation*}
Its zero-Dirichlet realization is invertible because that of $L$ is. For a solution $L^\star\widetilde{w}=0$ with boundary value in $C_c^\infty(\Gamma)$, integration by parts in \eqref{eq:higher-order-diff} gives
\begin{equation*}
 0=\int_{\partial M}\widetilde{w}\,\nu_i
 (D^{ij}\partial_jz+B^iz+T_K^i)
 +\int_M zL^\star\widetilde{w}\,dx
 -\int_M T_K^i\partial_i\widetilde{w}\,dx .
\end{equation*}
The boundary integral vanishes on $\Gamma$ by the measured identity and elsewhere because $\widetilde{w}=0$. The volume term containing $z$ also vanishes, and therefore
\begin{equation*}
 \int_M(\delta Q_K \mathcal{J})^{i,\alpha_1\cdots\alpha_K}
 \prod_{r=1}^K P_{\alpha_r}w^{h_r}\,
 \partial_i\widetilde{w}\,dx=0.
\end{equation*}
The notation denotes the intrinsic pairing of a tensor density with covectors. It is independent of the chart.

All components with at least one index $\alpha_r=0$, corresponding to a derivative with respect to $s$, are already determined. Indeed, since the lower-order identities hold as smooth functions of $s$ throughout the background interval,
\begin{equation}\label{eq:scalar-relations}
 (\delta Q_K \mathcal{J})^{i,\alpha_1\cdots\alpha_{K-1}0}
 =\partial_s(\delta Q_{K-1} \mathcal{J})^{i,\alpha_1\cdots\alpha_{K-1}}.
\end{equation}
Differentiating the induction hypothesis gives exactly the corresponding cases of \eqref{eq:Kth-order-coeff}. Symmetry in $\alpha_1,\ldots,\alpha_K$ covers every position of the zero index. Differentiating the boundary condition gives \eqref{eq:varphi-K-boundary}.

The constant direction is available by differentiation in $s$. For $h_1\equiv1$, we have $w^{h_1}=1$, $P_0w^{h_1}=1$ and $P_jw^{h_1}=0$ for $j\geq1$. The just-recovered coefficients with a zero index therefore give
\begin{equation*}
 T_K^i[1,h_2,\ldots,h_K]
 =\partial_j\left(\varphi_{(K)}^{ij}\prod_{r=2}^K w^{h_r}\right).
\end{equation*}
The cancellation above and zero-Dirichlet uniqueness in \eqref{eq:higher-order-diff} imply
\begin{equation*}
 w_1^{1,h_2,\ldots,h_K}-w_2^{1,h_2,\ldots,h_K}=0.
\end{equation*}
This compares the two mixed variations. It does not assume that either vanishes individually. Equality for arbitrary $K$-tuples will follow after the remaining coefficient recovery.

The remaining step recovers the coefficients involving only $\xi$ derivatives, namely $\partial_\xi^K\mathcal{J}_1-\partial_\xi^K\mathcal{J}_2$ at $(x,s,0)$, and then all order-$K$ solution variations. This alternation of coefficient recovery and solution comparison is also present in \cite{chen2026calderon}, where the background metric is fixed and the unknown is a scalar factor. The additional rank induction for that scalar integral identity cannot be imported as an injectivity statement for arbitrary anisotropic tensor coefficients.

Subtract the known current $\partial_j(\varphi_{(K)}^{ij}\prod_{r=1}^K w^{h_r})$ from $T_K$. The same cancellation makes its pairing with $d\widetilde{w}$ zero, leaving only the $K$ derivatives with respect to $\xi$. Fix the real solutions $w^{h_3},\ldots,w^{h_K}$ and set
\begin{equation*}
 C^{ijk}:=(\delta Q_K \mathcal{J})^{i,jk \alpha_3\cdots \alpha_K}\prod_{r=3}^K P_{\alpha_r}w^{h_r}
=(\delta Q_K \mathcal{J})^{i,jk j_3\cdots j_K}\prod_{r=3}^K\partial_{j_r}w^{h_r},
\end{equation*}
with the empty product understood as $1$ when $K=2$. The second equality follows from \eqref{eq:scalar-relations}: since $j,k\geq1$, any additional zero index gives a vanishing component. This tensor satisfies $C^{ijk}=C^{ikj}$ and \eqref{eq:int-id-3-1} for all smooth solutions $w_1,w_2\in\ker L$ and $\widetilde{w}\in\ker L^\star$ whose boundary traces belong to $C_c^\infty(\Gamma)$. Thus Proposition \ref{prop:three-gradient-boundary} applies.

\subsection{Injectivity of the three-gradient integral transform}
\label{subsec:three-gradient}

We prove Proposition \ref{prop:three-gradient-boundary}, then combine it with Runge approximation to recover the remaining coefficients involving only $\xi$ derivatives and close the simultaneous induction.

\begin{proof}[Proof of Proposition \ref{prop:three-gradient-boundary}]
We first extract a pointwise boundary identity, prove that its algebraic kernel is trivial, and then repeat the argument for higher-order normal jets.

\emph{Boundary symbol identity.} By complex multilinearity, \eqref{eq:int-id-3-1} also holds for complex-valued solutions. Fix $x_0\in\Gamma$ and a boundary coordinate neighborhood $U\Subset\Gamma$ containing it. Use boundary normal coordinates over $U$, with inward unit conormal $\nu=dx^n$. Choose nonzero tangential covectors $\xi,\eta$ with $\xi+\eta\neq0$ and real boundary cutoffs $\chi_1,\chi_2,\widetilde{\chi}$ in $C_c^\infty(U)$. Apply Proposition \ref{Prop-boundary-quasimodes} to scalar representatives of $L$ and $L^\star$, obtained by division by the same fixed positive density. The first two quasimodes correspond to $L$ and the third to $L^\star$. Their principal symbols agree, so the same phase covectors $p(\zeta)$ can be used for all three. This gives
\begin{align*}
 u^{\rm qm}_{1,\tau}&=e^{i\tau\Phi_1}a_{1,\tau},&
 d\Phi_1(x_0)&=p(\xi),\\
 u^{\rm qm}_{2,\tau}&=e^{i\tau\Phi_2}a_{2,\tau},&
 d\Phi_2(x_0)&=p(\eta),\\
 \widetilde{u}^{\rm qm}_\tau&=e^{i\tau\widetilde{\Phi}}\widetilde{a}_\tau,&
 d\widetilde{\Phi}(x_0)&=p(-\xi-\eta),
\end{align*}
where, at $x_0$,
\begin{equation*}
 p(\zeta):=\zeta+i|\zeta|_g\nu,\qquad
 a_{m,\tau}\sim\sum_{l\geq0}\tau^{-l} a_{m,l},\qquad
 \widetilde{a}_\tau\sim\sum_{l\geq0}\tau^{-l}\widetilde{a}_l,
\end{equation*}
Here $m=1,2$, and the leading boundary amplitudes are $\chi_1,\chi_2,\widetilde{\chi}$, respectively. Their tangential oscillations cancel. Correct these quasimodes to exact solutions by setting $w_{m,\tau}=u^{\rm qm}_{m,\tau}+r_{m,\tau}$ and $\widetilde{w}_\tau=\widetilde{u}^{\rm qm}_\tau+\widetilde{r}_\tau$, where
\begin{align*}
 Lr_{m,\tau}&=-Lu^{\rm qm}_{m,\tau},&
 r_{m,\tau}|_{\partial M}&=0,\\
 L^\star\widetilde{r}_\tau&=-L^\star\widetilde{u}^{\rm qm}_\tau,&
 \widetilde{r}_\tau|_{\partial M}&=0.
\end{align*}
For every prescribed $N$, the WKB construction and the Dirichlet elliptic estimate give, for each fixed Sobolev index $\ell$,
\begin{equation*}
 \|r_{1,\tau}\|_{H^\ell}+\|r_{2,\tau}\|_{H^\ell}
 +\|\widetilde{r}_\tau\|_{H^\ell}=O(\tau^{-N}).
\end{equation*}
Take $\ell>n/2+1$ and $N\geq1$, so the corrections are $O(\tau^{-N})$ in $C^1$. Write $u_{m}=u^{\rm qm}_{m,\tau}$, $r_{m}=r_{m,\tau}$, $\widetilde{u}=\widetilde{u}^{\rm qm}_\tau$, and $\widetilde{r}=\widetilde{r}_\tau$. The correction remainder is
\begin{align*}
 I_{\mathrm{corr}}
 &:=I(C;\widetilde{w}_\tau;w_{1,\tau},w_{2,\tau})-I(C;\widetilde{u};u_1,u_2)\\
 &=I(C;\widetilde{r};u_1,u_2)
   +I(C;\widetilde{u};r_1,u_2)+I(C;\widetilde{u};u_1,r_2)\\
 &\quad+I(C;\widetilde{r};r_1,u_2)
   +I(C;\widetilde{r};u_1,r_2)+I(C;\widetilde{u};r_1,r_2)
   +I(C;\widetilde{r};r_1,r_2).
\end{align*}
Since each quasimode gradient is $O(\tau)$ in $C^0$, the three terms with one correction are $O(\tau^{2-N})$, those with two corrections are $O(\tau^{1-2N})$, and the last is $O(\tau^{-3N})$. Thus $I_{\mathrm{corr}}=O(\tau^{2-N})$. For the $r$-th normal jet, choosing $N>r$ gives $\tau^{r-2}I_{\mathrm{corr}}=O(\tau^{r-N})=o(1)$.

Let $t=x^n$. After shrinking $U$ if necessary, write
\begin{equation*}
 \mu(x')=|\xi|_{g(x')}+|\eta|_{g(x')}
          +|\xi+\eta|_{g(x')}>0.
\end{equation*}
Their sum $\Theta=\Phi_1+\Phi_2+\widetilde{\Phi}$ satisfies $\Theta(x',0)=0$, $\Theta(x',t)=i\mu(x')t+O(t^2)$, and $\operatorname{Im}\Theta\geq ct$ on the common support. For every fixed $r\in\mathbb{Z}_{\geq0}$ and smooth amplitude $b_\tau=b_0(x',0)+O(t)+O(\tau^{-1})$, we have, uniformly on the compact tangential support,
\begin{equation}\label{eq:three-quasimode-scale}
 \int_0^\varepsilon t^r e^{i\tau\Theta(x',t)}b_\tau(x',t)\,dt
 =\frac{r!}{\tau^{r+1}\mu(x')^{r+1}}b_0(x',0)+O(\tau^{-r-2}).
\end{equation}
Indeed, $t=\theta/\tau$ gives the leading exponential moment $r!/\mu^{r+1}$. Replacing $\Theta$ by $i\mu t$ has error bounded by $C\tau t^2e^{-c\tau t}$, whose $t^r$-weighted integral is $O(\tau^{-r-2})$. The $O(t)+O(\tau^{-1})$ amplitude error has the same order. The exponential tail beyond $\varepsilon$ is negligible. The constants may depend on $r$, which remains fixed.

For the boundary value, take $r=0$. The three principal derivatives contribute $(i\tau)^3=-i\tau^3$. In the boundary integrals below, $p(\zeta)$ denotes the corresponding phase differential $p_{x'}(\zeta):=\zeta+i|\zeta|_{g(x')}\nu$, with the $x'$-dependence suppressed. Substitution in \eqref{eq:int-id-3-1} therefore gives
\begin{equation}\label{eq:three-quasimode-integral}
 0=\lim_{\tau\to\infty}\tau^{-2} I
 =-i\int_{\partial M}
 \frac{C^{ijk}p(-\xi-\eta)_ip(\xi)_jp(\eta)_k}
 {\mu}\,a_{1,0}a_{2,0}\widetilde{a}_0\,dx'.
\end{equation}
Here the smooth nonvanishing boundary coordinate density has been absorbed into one leading amplitude. Since the three leading amplitudes are arbitrary and may be localized at $x_0$, \eqref{eq:three-quasimode-integral} implies the pointwise identity
\begin{equation}\label{eq:boundary-symbol-C}
 C^{ijk}p(-\xi-\eta)_ip(\xi)_jp(\eta)_k=0
\end{equation}
at $x_0$.

\emph{Algebraic injectivity.} We now prove that \eqref{eq:boundary-symbol-C} has no algebraic kernel when $n\geq3$. Let the boundary tangential space be $V=\nu^\perp$, so $\dim V\geq2$, and we can write $C(r;x,y)=C^{ijk}r_ix_jy_k$. For tangential arguments set
\begin{align*}
 A(z;x,y)&:=C(z;x,y),&
 B(x,y)&:=C(\nu;x,y),\\
 D(z,y)&:=C(z;\nu,y),&
 E(y)&:=C(\nu;\nu,y),\\
 F(z)&:=C(z;\nu,\nu),&
 G&:=C(\nu;\nu,\nu).
\end{align*}
Let $u,v\in V$ be nonparallel unit covectors, put $\eta=sv$ with $s>0$, and set
\begin{equation*}
 q(s):=|u+sv|_g=\sqrt{1+2s\langle u,v\rangle_g+s^2}.
\end{equation*}
After removing the common factor $s$, the substituted symbol identity is
\begin{equation*}
 0=C\bigl(-(u+sv)+iq(s)\nu;\,u+i\nu;\,v+i\nu\bigr).
\end{equation*}
Its real and imaginary parts are
\begin{align}
 &[A(u;u,v)-F(u)]
 +s[A(v;u,v)-F(v)]
 +q(s)[E(u)+E(v)]=0,\label{eq:C-algebra-real}\\
 &q(s)[B(u,v)-G]
 -D(u,u+v)-sD(v,u+v)=0.\label{eq:C-algebra-imag}
\end{align}
We use the elementary fact that, when $|\langle u,v\rangle|<1$,
\begin{equation*}
 a+bs+cq(s)=0\quad\text{for every }s>0
 \quad\Longrightarrow\quad a=b=c=0.
\end{equation*}
Indeed, $q''(s)=(1-\langle u,v\rangle^2)/q(s)^3>0$. Differentiating the identity twice gives $c=0$, and then $a=b=0$.

Applying this fact separately to \eqref{eq:C-algebra-real} and \eqref{eq:C-algebra-imag} gives
\begin{align}
 A(u;u,v)&=F(u),&
 A(v;u,v)&=F(v),&
 E(u)+E(v)&=0,\label{eq:C-algebra-consequences-real}\\
 D(u,u+v)&=0,&
 D(v,u+v)&=0,&
 B(u,v)&=G.\label{eq:C-algebra-consequences-imag}
\end{align}
Replacing $v$ by $-v$ gives, in the same order,
\begin{align*}
 -A(u;u,v)&=F(u),&
 A(v;u,v)&=-F(v),&
 E(u)-E(v)&=0,\\
 D(u,u-v)&=0,&
 D(v,u-v)&=0,&
 -B(u,v)&=G.
\end{align*}
Comparing these relations with
\eqref{eq:C-algebra-consequences-real}--%
\eqref{eq:C-algebra-consequences-imag} yields
\begin{equation*}
 E(u)=E(v)=F(u)=F(v)=B(u,v)=G=0,
 \qquad A(u;u,v)=A(v;u,v)=0,
\end{equation*}
and
\begin{equation*}
 D(u,u)=D(u,v)=D(v,u)=D(v,v)=0.
\end{equation*}
Varying the nonparallel unit covectors and using homogeneity and continuity gives $E=F=B=G=D=0$ and $A(u;u,v)=0$ for all $u,v\in V$. Polarizing the last identity in the repeated argument gives
\begin{equation*}
 0=A(x+z;x+z,y)=A(x;z,y)+A(z;x,y),
\end{equation*}
so $A$ is antisymmetric in its first two indices. It is symmetric in its last two indices, and Lemma \ref{lmm:LAsym} therefore gives $A=0$. Thus all components of $C$ vanish at the boundary point.

\emph{Higher-order normal jets.} For the normal derivatives, assume inductively that $\partial_\nu^l C=0$ on $\Gamma$ for $l<r$. In the normal coordinate $t=x^n$,
\begin{equation*}
 C(x',t)=\frac{t^r}{r!}\partial_\nu^rC(x',0)+t^{r+1}C_{r,\mathrm{rem}}(x',t),
\end{equation*}
where $C_{r,\mathrm{rem}}$ is smooth and bounded on the fixed coordinate neighborhood.
Apply \eqref{eq:three-quasimode-scale} with this $r$. Its factor $r!$ cancels the Taylor coefficient $1/r!$, and the three principal derivatives give the leading order $ \tau^3\tau^{-r-1}=\tau^{2-r} $.
Multiplying by $\tau^{r-2}$ and passing to the limit yields
\begin{equation*}
 0=\lim_{\tau\to\infty}\tau^{r-2}I = -i\int_{\partial M}
 \frac{(\partial_\nu^rC)^{ijk}p(-\xi-\eta)_ip(\xi)_jp(\eta)_k}
 {\mu^{r+1}}\,a_{1,0}a_{2,0}\widetilde{a}_0\,dx'.
\end{equation*}
Localizing the arbitrary leading amplitudes and multiplying by the nonzero factor $\mu^{r+1}$ gives the same pointwise algebraic identity \eqref{eq:boundary-symbol-C} with $C$ replaced by $\partial_\nu^rC$. To identify the remaining errors, expand each quasimode derivative as
\begin{equation*}
 du^{\rm qm}_{m,\tau}
 =e^{i\tau\Phi_m}\bigl(i\tau a_{m,\tau}\,d\Phi_m+da_{m,\tau}\bigr),
\end{equation*}
with the analogous formula for the transposed equation. Terms containing at least one $da_{m,\tau}$ have integrands bounded by $C_r\tau^2t^re^{-c\tau t}$. The coefficient remainder $t^{r+1}C_{r,\mathrm{rem}}$ contributes at most $C_r\tau^3t^{r+1}e^{-c\tau t}$. After integration and the scaling $\tau^{r-2}$, these contributions are
\begin{equation*}
 \tau^{r-2} O\bigl(\tau^2\tau^{-r-1}\bigr)=O(\tau^{-1}),\qquad
 \tau^{r-2} O\bigl(\tau^3\tau^{-r-2}\bigr)=O(\tau^{-1}),
\end{equation*}
respectively. In the term with three phase derivatives, replacing each amplitude and phase differential by its leading boundary value produces an amplitude difference $O(t)+O(\tau^{-1})$. The corresponding tangential covectors are $\xi$, $\eta$, and $-\xi-\eta$, respectively. Its integrand is bounded by $C_r(\tau^3t^{r+1}+\tau^2t^r)e^{-c\tau t}$. Replacing $e^{i\tau\Theta}$ by $e^{-\tau\mu t}$ contributes at most $C_r\tau^4t^{r+2}e^{-c\tau t}$. Both normalized integrals are $O(\tau^{-1})$, as in \eqref{eq:three-quasimode-scale}. The exact-solution corrections are $O(\tau^{r-N})=o(1)$ for $N>r$. All tangential derivatives of the lower normal jets vanish by the induction hypothesis. Thus the same algebraic injectivity gives $\partial_\nu^rC=0$ and closes the induction.

\emph{Analytic continuation.} Finally, on the open boundary subset $\Sigma$ where $C$ is analytic, all normal jets and their tangential derivatives vanish. Its convergent Taylor series is therefore zero, so $C$ vanishes on a nonempty interior open set. The analytic identity theorem on the connected interior $M^\circ$ gives $C\equiv0$ there, and continuity gives the same conclusion on $\partial M$.
\end{proof}

We now apply Proposition \ref{prop:three-gradient-boundary} to the coefficients involving only $\xi$ derivatives. For $K=2$, these form the analytic tensor $C^{ijk}=(\delta Q_2 \mathcal{J})^{i,jk}$, with $j,k\in\{1,\ldots,n\}$, so the proposition gives $C^{ijk}=(\delta Q_2 \mathcal{J})^{i,jk}=0$. For $K>2$, choose nonempty open subsets $\Sigma,\Gamma_0\Subset\Gamma$ with disjoint closures, and take real $h_3,\ldots,h_K\in C_c^\infty(\Gamma_0)$. Keep their solutions fixed, independently of the boundary-layer parameter $\tau$. The contracted tensor
\begin{equation*}
 C^{ijk}=(\delta Q_K \mathcal{J})^{i,jk j_3\cdots j_K}
                 \prod_{r=3}^K\partial_{j_r}w^{h_r}
\end{equation*}
is smooth on $M$, analytic in $M^\circ$, and analytic up to $\Sigma$: condition \eqref{condition:regularity} makes the pure $\xi$ derivatives analytic in $x$ at the fixed $s$, and these fixed solutions have zero Dirichlet data near $\Sigma$, so Lemma \ref{lmm:analytic-boundary} applies there. Proposition \ref{prop:three-gradient-boundary} gives $C=0$ for every choice of the fixed boundary data $h_3,\ldots,h_K\in C_c^\infty(\Gamma_0)$.

At a fixed interior point, Lemma \ref{lmm:Runge-jets} with boundary data supported in $\Gamma_0$ realizes arbitrary gradients of each fixed solution. By multilinearity, every component $(\delta Q_K \mathcal{J})^{i,j_1\cdots j_K}$ with $j_1,\ldots,j_K\in\{1,\ldots,n\}$ is therefore zero. This proves the remaining case of \eqref{eq:Kth-order-coeff} for every $s$.

Now \eqref{eq:invisible-product} holds also for $l=K$, so the source $T_K$ in \eqref{eq:higher-order-diff} is divergence-free. The source cancellation holds for arbitrary smooth boundary directions, so zero-Dirichlet uniqueness gives $z=0$ on $M$. This completes both parts of the simultaneous induction without assuming equality of the $K$-th solution variations in advance.

\begin{corollary}\label{cor:higher-order-invisible}
Assume the hypotheses of Theorem \ref{thm:main} and make the fixed identification of Theorem \ref{thm:first-order}. For every $s\in\mathbb{R}$ and $K\geq1$, equations \eqref{eq:Kth-order-coeff} and \eqref{eq:varphi-K-boundary} hold, with $\delta Q_K \mathcal{J}=Q_K\mathcal{J}_1-Q_K\mathcal{J}_2$ evaluated at $(x,s,0)$ and $\varphi_{(K)}=\partial_s^{K-1}\varphi_{(1)}$ as in \eqref{eq:def-varphi-K}. In particular, for arbitrary $v_1,\ldots,v_K\in C^\infty(M)$,
\begin{equation*}
 (\delta Q_K \mathcal{J})^{i,\alpha_1\cdots\alpha_K}
 \prod_{r=1}^K P_{\alpha_r}v_r
 =
 \partial_j\left(\varphi_{(K)}^{ij}
 \prod_{r=1}^K v_r\right).
\end{equation*}
The two solution maps also have equal variations of every order at each constant background, in the notation \eqref{def-lin-sol}.
\end{corollary}
The simultaneous induction in this section proves the corollary. The next subsection reconstructs the current from these identities and verifies the zeroth-order conditions needed for Theorem \ref{thm:main}.

\subsection{Completion of the nonlinear recovery}
\label{subsec:nonlinear-completion}

We now use analyticity in $\xi$ to reconstruct the full current at each constant background and complete the proof of Theorem \ref{thm:main}.

\begin{proof}[Proof of Theorem \ref{thm:main}]
Theorem \ref{thm:first-order} supplies a single $\Gamma$-fixing analytic diffeomorphism $\Phi$. After replacing $\mathcal{J}_2$ by $H_\Phi\mathcal{J}_2$, which preserves \eqref{condition:regularity} since $\Phi$ is analytic and independent of $s$, Corollary \ref{cor:higher-order-invisible} gives the coefficient identities at every constant background. Set $\varphi=\varphi_{(1)}$. This field is smooth in $(x,s)$, analytic in $x$ for each fixed $s$, and satisfies
\begin{equation*}
 \varphi^{ij}=-\varphi^{ji},\qquad
 \nu_i\varphi^{ij}|_\Gamma=0.
\end{equation*}
We recover the full current difference from this field and the coefficients involving only $\xi$ derivatives.

Fix $x$ and $s$. The first-order identity and Corollary \ref{cor:higher-order-invisible} give
\begin{equation*}
 \partial_{\xi_j}(\mathcal{J}_1-\mathcal{J}_2)^i(x,s,0)=\varphi^{ij}(x,s),
 \qquad
 \partial_\xi^K(\mathcal{J}_1-\mathcal{J}_2)(x,s,0)=0,\quad K\geq2.
\end{equation*}
Analyticity in $\xi$ therefore makes $(\mathcal{J}_1-\mathcal{J}_2)^i-\varphi^{ij}\xi_j$ constant near $\xi=0$. The identity theorem extends this conclusion to the connected fiber $T_x^*M$. Since the coefficient identities hold at every $s$, we obtain
\begin{equation*}
 (\mathcal{J}_1-\mathcal{J}_2)^i(x,s,\xi)
 =\mathcal{I}^i(x,s)+\varphi^{ij}(x,s)\xi_j,
\end{equation*}
where $\mathcal{I}(x,s)=(\mathcal{J}_1-\mathcal{J}_2)(x,s,0)$ is smooth in $(x,s)$ and analytic in $x$ for each fixed $s$. The first-order identity involving the $s$ derivative gives
\begin{equation*}
 \partial_s\mathcal{I}^i=\partial_j\varphi^{ij}.
\end{equation*}
In particular, the normalization $\mathcal{I}(x,0)=0$ yields
\begin{equation*}
 \mathcal{I}^i(x,s)=\int_0^s\partial_j\varphi^{ij}(x,t)\,dt.
\end{equation*}
Finally, the equations at the constant solutions $u_m^s\equiv s$ and equality of $N_{\mathcal{J}_m}^\Gamma(s,0)$ give, for every $s$,
\begin{equation*}
 \partial_i\mathcal{I}^i=0\quad\text{in }M,\qquad
 \nu_i\mathcal{I}^i=0\quad\text{on }\Gamma.
\end{equation*}
These identities and the antisymmetry and boundary condition of $\varphi$ are precisely \eqref{condition:invisible}. Restoring the original notation for $\mathcal{J}_2$ gives the asserted decomposition with $H_\Phi\mathcal{J}_2$.
\end{proof}

\appendix

\section{A Hall conductor and the antisymmetric ambiguity}
\label{app:hall-example}

We illustrate the antisymmetric obstruction in a Hall-effect-type constitutive model with full boundary data, taking $\Gamma=\partial M=\partial\Omega$ throughout this appendix.

Let a conductor occupy a domain $\Omega\subset\mathbb{R}^3$ with symmetric positive definite zero-field conductivity $\sigma=\sigma^T>0$. The static magnetic field $\mathbf{B}$ is held fixed as the boundary value of $u$ varies, but need not be known. We use $\mathbf{E}=\nabla u$, so $u$ is the negative of the conventional electrostatic potential. Write $ R_{ijk}=\epsilon_{ijl}\mathcal{R}^l_k  $ for the Hall tensor, and $\mathbf{B}^l_{\mathrm{eff}}=\mathcal{R}^l_k\mathbf{B}^k$ for the effective Hall field. The Hall constitutive relation is
\begin{equation}\label{eq:Hall-effect}
 \mathcal{J}=\sigma(\mathbf{E}+\mathcal{J}\times\mathbf{B}_{\mathrm{eff}}),
 \quad\text{or equivalently } \mathbf{E}=(\sigma^{-1}-W)\mathcal{J},~\text{where } W_{ij}=\epsilon_{ijl}\mathbf{B}_{\mathrm{eff}}^l.
\end{equation}
This is a local resistivity model retaining the term linear in the magnetic field. See \cite{LANDAU198486} for a standard textbook treatment of the physical background. All constitutive relations are pointwise in $x$, and the material parameters and magnetic field may vary spatially.

For the isotropic conductors considered below, take $\sigma=\overline{\sigma}I$ and $\mathcal{R}=R_HI$ with $\overline{\sigma},R_H>0$. The effective Hall field and current law are then
\begin{equation*}
 \mathbf{b}=R_H\mathbf{B},\qquad
 W\mathbf{v}=\mathbf{v}\times\mathbf{b},\qquad
 \mathcal{J}(x,\xi)=(\overline{\sigma}^{-1}I-W)^{-1}\xi.
\end{equation*}
The potential satisfies $\nabla\cdot\mathcal{J}(x,\nabla u)=0$ with prescribed boundary value $f=u|_{\partial\Omega}$. The full DN map \eqref{def:DN} records the resulting normal current on $\partial\Omega$. The magnetic flux through a boundary component $\Gamma_j$ is $\int_{\Gamma_j}\mathbf{B}\cdot\mathbf{n}\,dS$, where $\mathbf{n}$ is the outward unit normal.

These currents are independent of $s$ and vanish at $\xi=0$, so the invisible current reduces to
\begin{equation*}
    \delta\mathcal{J}^i=\varphi^{ij}(x)\xi_j,
\end{equation*}
with $\varphi$ antisymmetric, divergence-free, and tangential to the boundary. For the isotropic model above, this difference corresponds to a weighted difference of the effective Hall fields. The following proposition gives geometric and divergence conditions under which this ambiguity vanishes.

\begin{proposition}[Recovery of the effective Hall field in fixed coordinates]\label{prop:hall-recovery}
Let $\Omega\subset\mathbb{R}^3$ be a bounded simply connected domain with real analytic boundary. For $m=1,2$, let $\overline{\sigma}_m,R_{H,m}>0$ and $\mathbf{B}_m$ be real analytic on $\overline{\Omega}$, with $\nabla\cdot\mathbf{B}_m=0$, and let $\mathcal{J}_m$ be the corresponding isotropic current laws above, with $\mathbf{b}_m=R_{H,m}\mathbf{B}_m$. Suppose that the full DN maps agree and that the identification in Theorem \ref{thm:main} can be taken to be the identity in the original Euclidean coordinates. Assume either that $\partial\Omega$ is connected or that each $\mathbf{B}_m$ has zero magnetic flux through every boundary component. Then
\begin{equation*}
 \overline{\sigma}_1=\overline{\sigma}_2,\qquad
 \mathbf{b}_1=\mathbf{b}_2.
\end{equation*}
Thus a known positive Hall coefficient determines the magnetic field. If the same magnetic field is known instead, the Hall coefficient is determined wherever that field is nonzero. Without either prior specification, the factors are not separately determined, since $(R_H,\mathbf{B})\mapsto(cR_H,\mathbf{B}/c)$ for any constant $c>0$ leaves $ \mathbf{b} $ unchanged.
\end{proposition}

\begin{proof}
Direct matrix algebra applied to \eqref{eq:Hall-effect} gives
\begin{equation*}
 \mathcal{J}=(I-\sigma W)^{-1}\sigma\mathbf{E}
 =(\gamma+\varphi)\mathbf{E},\qquad
 \gamma=(\sigma^{-1}-W\sigma W)^{-1},\qquad
 \varphi=\gamma W\sigma.
\end{equation*}
Here $\gamma$ is symmetric positive definite and $\varphi$ is antisymmetric.

For the isotropic specialization above, these identities give
\begin{equation*}
 \lambda_1(\gamma)=\lambda_2(\gamma)
 =\frac{\overline{\sigma}}{1+\overline{\sigma}^2|\mathbf{b}|^2},
 \qquad \lambda_3(\gamma)=\overline{\sigma},\qquad
 \varphi\mathbf{v}
 =-\frac{\overline{\sigma}^2}{1+\overline{\sigma}^2|\mathbf{b}|^2}
       \mathbf{b}\times\mathbf{v} .
\end{equation*}
Under the fixed-coordinate assumption in Proposition \ref{prop:hall-recovery}, Theorem \ref{thm:main} gives $\gamma_1=\gamma_2=:\gamma$. The extreme eigenvalues then give
\begin{equation*}
 \overline{\sigma}_1=\overline{\sigma}_2
 =\lambda_{\max}(\gamma)=:\overline{\sigma},\qquad
 |\mathbf{b}_1|^2=|\mathbf{b}_2|^2
 =\frac{1}{\overline{\sigma}\lambda_{\min}(\gamma)}
  -\frac{1}{\overline{\sigma}^2}.
\end{equation*}
These formulae include the case $\mathbf{b}_m=0$, when the eigenvalues coincide. Write the remaining antisymmetric difference as
\begin{equation*}
    (\varphi_1-\varphi_2)\mathbf{v}=(\delta\varphi)\mathbf{v}=:-\widehat{\mathbf{B}}\times\mathbf{v}.
\end{equation*}
The difference $\delta\varphi$ is invisible precisely when
\begin{equation*}
 \nabla\times\widehat{\mathbf{B}}=0\quad\text{in }\Omega,\qquad
 \mathbf{n}\times\widehat{\mathbf{B}}=0\quad\text{on }\partial\Omega.
\end{equation*}

First suppose that $\partial\Omega$ is connected. Put $a=\overline{\sigma}\lambda_{\min}(\gamma)>0$ and $\widehat{\mathbf{B}}=a(\mathbf{b}_1-\mathbf{b}_2)$. Since $\Omega$ is simply connected, the preceding curl and boundary conditions give $\widehat{\mathbf{B}}=\nabla\psi$ with $\psi|_{\partial\Omega}=0$. Using $|\mathbf{b}_1|=|\mathbf{b}_2|$ and $\mathbf{b}_2=R_{H,2}\mathbf{B}_2$, we expand the square and apply Stokes' theorem in divergence form:
\begin{equation*}
 \begin{aligned}
 \int_\Omega\frac{|\widehat{\mathbf{B}}|^2}{aR_{H,2}}\,dx
 &=\int_\Omega\frac{a}{R_{H,2}}
 \left((\mathbf{b}_1-\mathbf{b}_2)\cdot(\mathbf{b}_1+\mathbf{b}_2)-2(\mathbf{b}_1-\mathbf{b}_2)\cdot\mathbf{b}_2\right)\,dx\\
 &=-2\int_\Omega\frac{a}{R_{H,2}}
 (\mathbf{b}_1-\mathbf{b}_2)\cdot\mathbf{b}_2\,dx\\
 &=-2\int_\Omega\widehat{\mathbf{B}}\cdot\mathbf{B}_2\,dx
 =-2\int_\Omega\nabla\psi\cdot\mathbf{B}_2\,dx\\
 &=
 -2\int_{\partial\Omega}\psi\,\mathbf{B}_2\cdot\mathbf{n}\,dS
 +2\int_\Omega\psi\,(\nabla\cdot\mathbf{B}_2)\,dx\\
 &=0.
 \end{aligned}
\end{equation*}
The boundary term vanishes because $\psi|_{\partial\Omega}=0$, and the volume term vanishes by the physical constraint $\nabla\cdot\mathbf{B}_2=0$. Since $aR_{H,2}>0$, this identity forces $\widehat{\mathbf{B}}=0$ and hence $\mathbf{b}_1=\mathbf{b}_2$. If $\partial\Omega$ is disconnected, $\psi$ is constant on each component $\Gamma_j$, so zero magnetic flux through every component again makes the boundary term vanish in the same Stokes identity.

Thus the physical divergence constraint, together with the stated geometric and boundary conditions, removes the antisymmetric obstruction and yields uniqueness of the effective Hall field $\mathbf{b}=R_H\mathbf{B}$ in the fixed Euclidean coordinates. A known positive $R_H$ determines $\mathbf{B}$. If the same $\mathbf{B}$ is known instead, the symmetric-part identities already determine $R_H=|\mathbf{b}|/|\mathbf{B}|$ wherever $\mathbf{B}\ne0$. If both factors are unknown, the rescaling $(R_H,\mathbf{B})\mapsto(cR_H,\mathbf{B}/c)$, $c>0$ constant, preserves the entire current law and the divergence constraint.
\end{proof}


\section{Poisson embeddings and global extension for elliptic PDOs}
\label{app:poisson-continuation}

We use the Runge theorem of \cite[Proposition~A.2]{lassas2020poisson} directly and record the short reductions needed for our operators. The embedding and continuation arguments follow \cite[Propositions~2.1--2.3, Theorem~3.3, and Appendix~B]{lassas2020poisson}, whose stated harmonic setting is replaced here by real elliptic partial differential operators with $L1=0$, and all boundary evaluations are placed in one fixed space $H^{-r}(\partial M)$.

\subsection{Runge realization of one-jets}\label{app:runge-jets}
For the corresponding prescribed-jet argument for $\Delta_g+q$, see \cite[Propositions~3.3--3.4 and Appendix~A.2]{carstea2026conformalpoisson}. We reduce the prescribed one-jet realization to the Runge approximation in \cite[Proposition~A.2]{lassas2020poisson}, using local analytic solvability to supply the desired jets.
\begin{proof}[Proof of Lemma \ref{lmm:Runge-jets}]
Choose disjoint small coordinate balls around the prescribed points with connected complement. The cited proposition applies to the present uniformly elliptic $L$ with invertible Dirichlet problem. Unique continuation for $L$ and its adjoint holds here. The same theorem is used on disjoint balls in the proof of Proposition~A.1(a) there. Interior elliptic estimates give $C^1$ approximation on smaller balls.

At each point, Cauchy--Kowalevski on an analytic noncharacteristic hypersurface produces a local solution with arbitrary prescribed value and first derivative.

Thus the image of the global solution space in the finite-dimensional space of one-jets is dense. It is a linear subspace and hence closed, so every prescribed finite collection of one-jets is attained. This is the finite-dimensional argument of \cite[Proposition~A.5]{lassas2020poisson}, with local analytic solvability supplying the jets for $L$.
\end{proof}

\subsection{Estimates for the compactified Poisson embedding}
\label{app:poisson-estimates}
We reduce the compactified embedding statement to the estimates and separation arguments of \cite[Propositions~2.1--2.3 and Appendix~B]{lassas2020poisson}, verifying them below for the present real elliptic operator with $L1=0$ and the fixed target space $H^{-r}(\partial M)$. See also \cite[Propositions~3.2, 3.5 and Lemma~3.8]{carstea2026conformalpoisson} for the corresponding regularity and embedding arguments for $\Delta_g+q$.
\begin{proof}[Proof of Lemma \ref{lmm:compactified-poisson}]
Choose $0<\alpha<\min\{1,r-(n-1)/2-2\}$. Sobolev embedding, the boundary extension and Schauder estimates \cite[Lemma~6.38 and Theorem~6.6]{gilbarg2001elliptic}, and the maximum principle for $L1=0$ give
\begin{equation*}
 \|w^h\|_{C^{2,\alpha}(M)}\leq C\|h\|_{H^r(\partial M)}.
\end{equation*}
This estimate extends the evaluation functionals to $H^r$ and makes $\overline{\mathscr{P}}_L$ norm-continuous on $M$. In collar coordinates $x=(y,t)$, integration of the normal derivative gives $\|\overline{\mathscr{P}}_L(x)-\delta_y\|_{H^{-r}}\leq Ct$. Tangential integration similarly gives $\|\delta_y-\delta_z\|_{H^{-r}}\leq C d_{\partial M}(y,z)$ within a boundary chart.

For interior analyticity, the local analytic estimates in \cite[Theorem~5.7.1, pp.~164--167]{morrey1966multiple} yield, on each compact interior coordinate ball $K$,
\begin{equation*}
 \|\partial^\beta w^h\|_{C^0(K)}
 \leq C_K^{|\beta|+1}|\beta|!\,\|h\|_{H^r(\partial M)}.
\end{equation*}
The constant is independent of $h$. Taking the supremum over the unit ball of $H^r$ in Taylor's formula gives a remainder bounded by $C_K(C_K|z|_1)^{N+1}$ on a smaller ball, where $|z|_1=\sum_j|z_j|$. Hence the Taylor series converges in $H^{-r}$, proving analyticity of the interior Poisson embedding.

Lemma \ref{lmm:Runge-jets} supplies the point separation and injective differential used in \cite[Propositions~2.2--2.3]{lassas2020poisson}. Boundary evaluations separate boundary points. To separate $x\in M^\circ$ from $y\in\partial M$, choose nonzero nonnegative boundary data supported away from $y$. The strong maximum principle gives $w^h(x)>0=h(y)$. Finally, a continuous injection from compact $M$ to the Hausdorff space $H^{-r}$ is a homeomorphism onto its image. Thus $\overline{\mathscr{P}}_L$ is a topological embedding of $M$, and its restriction $\mathscr{P}_L$ is a real analytic embedding of $M^\circ$.
For the local map, the estimates above restrict continuously to $E_\Gamma$. Runge approximation with data in $C_c^\infty(\Gamma)$ gives interior point separation and immersion. Boundary evaluations separate points of $\Gamma$, and nonnegative data supported away from a given boundary point separate it from every interior point by the strong maximum principle. The map vanishes on $\partial M\setminus\Gamma$ and is nonzero on $M^\circ\cup\Gamma$. If $\mathscr{P}_L^\Gamma(x_j)\to\mathscr{P}_L^\Gamma(x)$ with $x\in M^\circ\cup\Gamma$, compactness gives subsequential limits in $M$. None lies in the unobserved boundary, where the map is zero, and injectivity on $M^\circ\cup\Gamma$ forces every limit to be $x$. This proves continuity of the inverse and completes the local embedding assertion of Lemma \ref{lmm:compactified-poisson}.
\end{proof}

\subsection{Global extension via Poisson embeddings}
\label{app:poisson-globalization}
We reduce the global identification to the continuation argument of \cite[Theorem~3.3]{lassas2020poisson}, using the local Poisson embeddings for the present real elliptic operators. See also \cite[Lemma~3.9 and Theorem~5.3]{carstea2026conformalpoisson} for the corresponding coordinate and continuation arguments for scaled identifications of Schr\"odinger solutions.
\begin{proof}[Proof of Lemma \ref{lmm:poisson-continuation}]
Write $R_m=\mathscr{P}_{L_m}^\Gamma$, including its continuous extension to $M$. By Lemma \ref{lmm:compactified-poisson}, the boundary maps $F_p$ agree on overlaps, since they all satisfy $R_2=R_1\circ F_p$. Choose one connected collar at $\Gamma$ as a seed and let $\mathcal{D}\subset M^\circ$ be the maximal connected open set on which it continues to an analytic local diffeomorphism $\Phi_{\mathcal{D}}$ preserving all solutions with boundary data in $C_c^\infty(\Gamma)$. Injectivity of $R_1$ makes the continuation unique.

If $x_k\in\mathcal{D}$ tends to $x\in M^\circ$, compactness and continuity give a subsequential limit $\Phi_{\mathcal{D}}(x_k)\to y\in M$ with $R_2(x)=R_1(y)$. The point $y$ is interior: otherwise nonnegative nonzero data in $C_c^\infty(\Gamma\setminus\{y\})$ would give a solution positive at $x$ and zero at $y$. Choose, by Lemma \ref{lmm:Runge-jets}, boundary data $h_j\in C_c^\infty(\Gamma)$ giving an analytic solution coordinate map $W_1=(w_1^{h_1},\ldots,w_1^{h_n})$ near $y$, and set $W_2=(w_2^{h_1},\ldots,w_2^{h_n})$. Since $W_2(x)=W_1(y)$, the analytic map

\begin{equation*}
 \widetilde{\Phi}=W_1^{-1}\circ W_2
\end{equation*}
is defined near $x$. No invertibility of $W_2$ has yet been assumed. The continuous inverse of $R_1$ on $R_1(M^\circ\cup\Gamma)$ ensures that $\Phi_{\mathcal{D}}$ takes all nearby points of $\mathcal{D}$ into this same $W_1$ chart. Thus $\widetilde{\Phi}=\Phi_{\mathcal{D}}$ on a nonempty open part of a small connected neighborhood $U$ of $x$.

For every $h\in C_c^\infty(\Gamma)$, analytic continuation on $U$ now gives
\begin{equation*}
 w_2^h=(w_1^h\circ W_1^{-1})\circ W_2.
\end{equation*}
If $v\in\ker dW_2(x)$, these identities give $dw_2^h(x)v=0$ for every $h$. Lemma \ref{lmm:Runge-jets} implies $v=0$. Hence $\widetilde{\Phi}$ is a local diffeomorphism and extends the identification across $x$. Therefore $\mathcal{D}$ is closed, and connectedness gives $\mathcal{D}=M^\circ$.

Interchanging the operators gives a map $\Psi$ in the opposite direction. Poisson injectivity yields $\Psi\circ\Phi=\mathrm{id}$ and $\Phi\circ\Psi=\mathrm{id}$, so $\Phi$ is a diffeomorphism of the interiors. The same injectivity makes $\Phi$ agree with each $F_p$, giving a unique $\Gamma$-fixing real analytic extension to $M^\circ\cup\Gamma$. If the collar maps preserve orientation, connectedness makes $\Phi$ orientation-preserving throughout $M^\circ\cup\Gamma$.

\end{proof}

\section{Magnetic boundary determination}
\label{app:magnetic-boundary-jets}

Boundary determination from the magnetic Schr\"odinger DN map follows the symbol factorization of \cite[Section~8]{dossantosferreira2009limiting}, with the boundary normalization corrected in \cite[Lemma~2.1(b)]{LassasMatti2022TCpf}. The gauges first remove the freedom that the boundary symbol cannot distinguish. The odd and even parts of that symbol then recover successive magnetic and metric jets. A trace identity separates the electric potential from the remaining metric derivative. We finally use the local DN data to recover solution jets and return to the original gauges. This last step is needed for the solution-coordinate construction in Subsection \ref{subsec:analytic-extension}.

\begin{proof}[Proof of Lemma \ref{lmm:magnetic-boundary-jets}]
\emph{Normalization.} The hypotheses give equal density-valued magnetic Schr\"odinger measurements on $\Gamma$ for data in $\mathcal{H}_\Gamma$, and equal boundary metrics on $\Gamma$. We first record the normalization in its original conformal convention. By \cite[Lemma~2.1(b), (2.1)]{LassasMatti2022TCpf}, there is a positive smooth function $\widehat{c}$ such that, for $\widehat{g}=\widehat{c} g$,
\begin{equation*}
 \widehat{c}|_{\partial M}=1,\qquad
 \partial_{\nu_g}\widehat{c}|_{\partial M}=0,\qquad
 (\mathcal{L}_{\widehat{N}})^j\widehat{H}|_{\partial M}=0
 \quad(j\geq1).
\end{equation*}
Here $\widehat{N}$ is the inward $\widehat{g}$-unit normal extended along its normal geodesics, and $\widehat{H}=\Delta_{\widehat{g}}d_{\widehat{g}}(\,\cdot\,,\partial M)$. Apply this separately to $g_m$ and put $c_m=\widehat{c}_m^{-1}$. The constant boundary trace also makes the tangential derivatives of $\widehat{c}_m$ vanish, so
\begin{equation*}
 c_m>0,\qquad c_m|_{\partial M}=1,\qquad dc_m|_{\partial M}=0,
 \qquad \widetilde{g}_m=c_m^{-1}g_m.
\end{equation*}
Let $M_m$ denote the $m$th copy of $M$, equipped with $\widetilde{g}_m$, and choose a common boundary coordinate neighborhood $U\Subset\Gamma$ of the point under consideration. Use the same tangential coordinates $x'$ on $U$ and the respective inward $\widetilde{g}_m$-normal geodesics to define $\chi_m:\mathcal{C}\to M_m$, where $\mathcal{C}=U\times[0,\varepsilon)$. After shrinking $U$ if necessary, choose the common width $\varepsilon>0$ sufficiently small that both $\chi_m$ are diffeomorphisms onto their collar neighborhoods. In these boundary normal coordinates, write
\begin{equation*}
 \chi_m^*\widetilde{g}_m=dt^2+h_{m,ab}\,dx^a dx^b,\qquad
 G_m^{ab}=h_m^{ab},\qquad
 E_m=\tfrac{1}{2}h_{m,ab}\partial_tG_m^{ab}.
\end{equation*}
In these coordinates $\widehat{N}=\partial_t$ and
\begin{equation*}
 \widehat{H}_m
 =\frac{1}{\sqrt{\det h_m}}\partial_t\sqrt{\det h_m}
 =\tfrac{1}{2}G_m^{ab}\partial_t h_{m,ab}.
\end{equation*}
Differentiating $h_{m,ac}G_m^{cb}=\delta_a^{\ b}$ and taking the trace gives
\begin{equation*}
 G_m^{ab}\partial_t h_{m,ab}
 =-h_{m,ab}\partial_tG_m^{ab},\qquad E_m=-\widehat{H}_m.
\end{equation*}
Consequently $\partial_t^jE_m|_{t=0}=0$ for every $j\geq1$. These are the new metric's own boundary normal coordinates, as required by the corrected normalization.

The conformal choice fixes the higher normal derivatives of the trace part of the metric. Independently, a magnetic gauge removes the normal component of the one-form without changing the boundary measurement. Choose a purely imaginary $\psi_m$, zero on $\partial M$, by
\begin{equation*}
 \psi_m\circ\chi_m(x',t)
 =-\int_0^t(\chi_m^* A_m)_t(x',t')\,dt'
\end{equation*}
on a smaller collar, and extend it as a purely imaginary smooth function to $M$, preserving its boundary value. Then $\widetilde{A}_m=A_m+d\psi_m$ has zero normal component throughout that collar. Put $\alpha=(n-2)/4$ and
\begin{equation*}
 \rho_m=c_m^\alpha e^{-i\psi_m}>0,\qquad
 \widetilde{q}_m
 =c_m\bigl(q_m-c_m^\alpha\Delta_{g_m}(c_m^{-\alpha})\bigr).
\end{equation*}
The two identities in the proof of \cite[Proposition~8.2]{dossantosferreira2009limiting}, written with our notation $\mathcal{P}$, are
\begin{equation*}
 c^{(n+2)/4}\mathcal{P}_{g,A,q}(c^{-(n-2)/4}v)
 =\mathcal{P}_{c^{-1}g,A,c(q-q_c)}v,\qquad
 q_c=c^{(n-2)/4}\Delta_g(c^{-(n-2)/4}),
\end{equation*}
and
\begin{equation*}
 e^{-i\psi}\mathcal{P}_{g,A,q}(e^{i\psi}v)
 =\mathcal{P}_{g,A+d\psi,q}v.
\end{equation*}
Applying the magnetic identity after the conformal identity gives
\begin{align*}
 \widetilde{\mathcal{P}}_m v
 &:=\mathcal{P}_{c_m^{-1}g_m,A_m+d\psi_m,\widetilde{q}_m}v\\
 &=e^{-i\psi_m}c_m^{1+\alpha}\mathcal{P}_m
       (c_m^{-\alpha}e^{i\psi_m}v)
 =c_m\rho_m\mathcal{P}_m(\rho_m^{-1}v).
\end{align*}
The identity also preserves Dirichlet nondegeneracy. It preserves the density-valued magnetic Schr\"odinger map exactly. First,
\begin{equation*}
 \widetilde{\gamma}_m^{ij}
 =\sqrt{\det(c_m^{-1}g_m)}(c_m^{-1}g_m)^{ij}
 =c_m^{1-n/2}\gamma_m^{ij}=c_m^{-2\alpha}\gamma_m^{ij}.
\end{equation*}
Also $d\rho_m=\rho_m(\alpha\,d\log c_m-i\,d\psi_m)$, and hence
\begin{equation*}
 d_{\widetilde{A}_m}(\rho_m w)
 =\rho_m\bigl(d_{A_m}w+\alpha\,d\log c_m\,w\bigr).
\end{equation*}
Their density-valued magnetic conormal derivatives are unchanged because $c_m=\rho_m=1$ and $dc_m=0$ there. In particular, no vanishing normal derivative of $\rho_m$ is needed: its magnetic contribution cancels the $i\,d\psi_m$ term. The boundary metric is also unchanged, so division by its common volume density gives equal scalar magnetic Schr\"odinger maps, as used in the cited factorization.

\emph{Coefficient jets.} Work in the new boundary normal coordinates and pull back all normalized coefficients by $\chi_m$. We reproduce the factorization and its recursion before extracting the coefficient jets. Use left quantization and $D_x=-i\partial_x$. In \cite[Lemma~8.6]{dossantosferreira2009limiting}, $x_n$ and $g^{\alpha\beta}$ correspond to our $t$ and $G^{ab}$. We write $V$ for its zeroth-order coefficient $\widetilde{G}$. Suppressing the index $m$, set
\begin{align*}
 q_2&=G^{ab}\xi_a\xi_b,\\
 q_1&=-i\left(\tfrac{1}{2}G^{ab}\partial_a\log\det h
                    +\partial_aG^{ab}\right)\xi_b,\\
 V&=-\frac{i}{\sqrt{\det h}}\partial_a
       \left(\sqrt{\det h}\,G^{ab}\widetilde{A}_b\right)
       +G^{ab}\widetilde{A}_a\widetilde{A}_b+\widetilde{q}.
\end{align*}
Here $q_2,q_1$ are tangential symbols, distinct from the electric potential $\widetilde{q}$. Since $\widetilde{A}_t=0$, the operator is
\begin{equation*}
 \widetilde{\mathcal{P}}
 =D_t^2+iED_t+q_2(x,D_{x'})+q_1(x,D_{x'})
       +2G^{ab}\widetilde{A}_aD_b+V.
\end{equation*}
The original factorization and symbol equation \cite[(8.3)--(8.4)]{dossantosferreira2009limiting} read, in this notation,
\begin{equation*}
 \widetilde{\mathcal{P}}
 =(D_t+iE-iB)(D_t+iB)\quad\bmod\Psi^{-\infty},
\end{equation*}
\begin{equation*}
 \partial_tb-Eb+
 \sum_{\beta}\frac{\partial_\xi^\beta b\,D_{x'}^\beta b}{\beta!}
 =q_2+q_1+2G^{ab}\widetilde{A}_a\xi_b+V\quad\bmod S^{-\infty}.
\end{equation*}
Here $\beta$ ranges over tangential multi-indices. Writing $b\sim\sum_{\ell\leq1}b_\ell$, comparison of homogeneous degrees gives \cite[(8.5)--(8.8)]{dossantosferreira2009limiting}:
\begin{align*}
 b_1&=-\sqrt{q_2},\\
 b_0&=\frac{1}{2b_1}\left(
 -\partial_tb_1+Eb_1-\partial_\xi b_1\cdot D_{x'}b_1
       +q_1+2G^{ab}\widetilde{A}_a\xi_b\right),\\
 b_{-1}&=\frac{1}{2b_1}\left(
 -\partial_tb_0+Eb_0
 -\sum_{\substack{0\leq p,k\leq1\\p+k=|\beta|}}
   \frac{\partial_\xi^\beta b_p\,D_{x'}^\beta b_k}{\beta!}+V\right),\\
 b_{\ell-1}&=\frac{1}{2b_1}\left(
 -\partial_tb_\ell+Eb_\ell
 -\sum_{\substack{\ell\leq p,k\leq1\\p+k-|\beta|=\ell}}
   \frac{\partial_\xi^\beta b_p\,D_{x'}^\beta b_k}{\beta!}\right),
 \qquad\ell\leq-1.
\end{align*}
Each sum includes all displayed integers and all admissible $\beta$. The same lemma gives $\widetilde{\Lambda}=-B|_{t=0}+R$ with $R$ smoothing. Thus equality of the normalized DN maps makes every homogeneous component $b_{m,\ell}|_{t=0}$ agree.

Put $r=\sqrt{q_2}$ and $\omega=\xi/r$, where $\xi\ne0$. The principal symbol gives $G|_{t=0}$. The tangential part of $b_0$ is
\begin{equation*}
 R_0:=\frac{-\partial_\xi b_1\cdot D_{x'}b_1+q_1}{2b_1}
     =\frac{\partial_\xi r\cdot D_{x'}r-q_1}{2r},
\end{equation*}
which depends only on $G$ and its tangential derivatives. Since
\begin{equation*}
 \partial_tb_1=-\frac{(\partial_tG^{ab})\xi_a\xi_b}{2r},
 \qquad G^{ab}\omega_a\omega_b=1,
\end{equation*}
substitution into the displayed equation for $b_0$ gives
\begin{align*}
 b_0-R_0
 &=-\frac{(\partial_tG^{ab})\xi_a\xi_b}{4r^2}
      +\frac{E}{2}-\frac{G^{ab}\widetilde{A}_a\xi_b}{r}\\
 &=-\tfrac{1}{4} k^{ab}\omega_a\omega_b
       -G^{ab}\widetilde{A}_a\omega_b,\qquad
 k^{ab}:=\partial_tG^{ab}-2EG^{ab}.
\end{align*}
The odd and even parts in $\omega$ determine $\widetilde{A}_a$ and $k^{ab}$, respectively. Since
\begin{equation*}
 h_{ab}k^{ab}=-2(n-2)E,\qquad
 \partial_tG^{ab}=k^{ab}+2EG^{ab},
\end{equation*}
they also determine $E$ and $\partial_tG$. This step is nondegenerate for every $n\geq3$.

At higher orders, the even symbol initially determines a combination of a metric derivative and an electric-potential derivative, not either separately. The prescribed derivatives of $E_m$ supply the missing trace information. We make this simultaneous induction explicit. Let $\delta$ denote the difference between the two normalized coefficient sets, evaluated at $t=0$. Suppose, for $j\geq1$, that
\begin{equation*}
 \delta\partial_t^\ell G=0\ (0\leq\ell\leq j),\qquad
 \delta\partial_t^\ell\widetilde{A}_a=0\ (0\leq\ell\leq j-1),\qquad
 \delta\partial_t^\ell\widetilde{q}=0\ (0\leq\ell\leq j-2),
\end{equation*}
with the last range empty when $j=1$. To derive the new jets from the original recursion, introduce the combination used in \cite[Lemma~8.7]{dossantosferreira2009limiting},
\begin{equation*}
 l^{ab}:=\tfrac{1}{4}\partial_tk^{ab}+\widetilde{q}G^{ab}.
\end{equation*}
The displayed formula for $b_{-1}$ gives
\begin{equation*}
 b_{-1}=\frac{1}{2b_1}
 \left(l^{ab}\omega_a\omega_b
       +G^{ab}(\partial_t\widetilde{A}_a)\omega_b\right)+R_1.
\end{equation*}
Indeed, differentiating $b_0-R_0$ contributes $\tfrac{1}{4}(\partial_tk^{ab})\omega_a\omega_b +G^{ab}(\partial_t\widetilde{A}_a)\omega_b$ to $-\partial_tb_0$, while $V$ contributes $\widetilde{q}=\widetilde{q}G^{ab}\omega_a\omega_b$. All other terms form $R_1$, which depends only on $G,\partial_tG$, $\widetilde{A}$, and their tangential derivatives.

At subsequent orders, the part raising the maximal normal derivative is
\begin{equation*}
 b_{-j-1}=-\frac{1}{2b_1}\partial_tb_{-j}
       +\text{terms involving lower normal jets}.
\end{equation*}
Iterating this identity yields the formula in the proof of \cite[Lemma~8.7]{dossantosferreira2009limiting}, for $j\geq1$,
\begin{equation*}
 b_{-j}=-\left(-\frac{1}{2b_1}\right)^j
 \left[(\partial_t^{j-1}l^{ab})\omega_a\omega_b
       +G^{ab}(\partial_t^j\widetilde{A}_a)\omega_b\right]+R_j.
\end{equation*}
The remainder depends only on $G,\partial_tG$, derivatives of $l$ through order $j-2$, and derivatives of $\widetilde{A}$ through order $j-1$, together with tangential derivatives. The $l$ range is empty when $j=1$. These remainders therefore agree under the induction hypothesis. Since every positive normal derivative of $E_m$ vanishes, the product rule gives
\begin{align*}
 \delta\partial_t^{j-1}l^{ab}
 &=\tfrac{1}{4}\delta\partial_t^j(\partial_tG^{ab}-2EG^{ab})
       +\delta\partial_t^{j-1}(\widetilde{q}G^{ab})\\
 &=\tfrac{1}{4}\delta\partial_t^{j+1}G^{ab}
       +G^{ab}\delta\partial_t^{j-1}\widetilde{q}.
\end{align*}
Using $b_1=-r$ and subtracting the two formulae for $b_{-j}$ now yields
\begin{equation*}
 0=\delta b_{-j}=-(2r)^{-j}
 \left[
  \tfrac{1}{4}\delta(\partial_t^{j+1}G^{ab})\omega_a\omega_b
  +\delta(\partial_t^{j-1}\widetilde{q})
  +G^{ab}\delta(\partial_t^j\widetilde{A}_a)\omega_b
 \right].
\end{equation*}

The odd part gives $\delta\partial_t^j\widetilde{A}_a=0$, and the even part gives
\begin{equation*}
 \delta(\partial_t^{j+1}G^{ab})
 =-4\,\delta(\partial_t^{j-1}\widetilde{q})\,G^{ab}.
\end{equation*}
On the other hand, differentiating $2E=h_{ab}\partial_tG^{ab}$ gives
\begin{equation*}
 2\partial_t^jE
 =h_{ab}\partial_t^{j+1}G^{ab}
  +\sum_{\ell=1}^j\binom{j}{\ell}
       (\partial_t^\ell h_{ab})(\partial_t^{j+1-\ell}G^{ab}).
\end{equation*}
The lower metric jets agree, and $\partial_t^jE_m=0$, so $h_{ab}\delta(\partial_t^{j+1}G^{ab})=0$. Taking the trace of the preceding even-part identity, with $h_{ab}G^{ab}=n-1$, therefore gives $(n-1)\delta(\partial_t^{j-1}\widetilde{q})=0$. This recovers both new jets and closes the induction. All equalities hold on the open boundary subset $U$ under consideration, so tangential differentiation recovers all mixed jets as well. The remaining components are fixed by $\widetilde{g}_{tt}=1$, $\widetilde{g}_{ta}=0$, and $\widetilde{A}_t=0$. Hence
\begin{equation*}
 \widetilde{\mathcal{P}}_2^{\mathcal{C}}
 -\widetilde{\mathcal{P}}_1^{\mathcal{C}}=O(t^\infty).
\end{equation*}
The parity and trace arguments are algebraic over $\mathbb{C}$, consistent with the complex magnetic convention of \cite[Section~1]{dossantosferreira2009limiting}. They do not require a unitary connection.

\emph{Solution jets and removal of the gauges.} We use the normal-derivative recursion of \cite[Lemma~3.4]{LassasMatti2022TCpf}. Its proof applies to operators with the same coefficient jets, a nonzero $\partial_t^2$ coefficient, and equal Cauchy traces, including the magnetic lower-order terms here. The local DN data supply the required Cauchy data for $h\in\mathcal{H}_\Gamma$. Set $c_m^{\mathcal{C}}=c_m\circ\chi_m$, $\rho_m^{\mathcal{C}}=\rho_m\circ\chi_m$, and $\widetilde{U}_m^h=\rho_m^{\mathcal{C}}U_m^h$. These solutions have the same Dirichlet value $h|_U$. In the normal magnetic gauge their outward density-valued magnetic conormal derivative is $-\sqrt{\det h_m}\,\partial_t\widetilde{U}_m^h\,d\sigma_0$, where $d\sigma_0$ is the positive coordinate density in $x'$. Equality of the measured DN maps on $\Gamma$, including their smoothing parts, therefore gives
\begin{equation*}
 \widetilde{U}_2^h-\widetilde{U}_1^h
 =\partial_t(\widetilde{U}_2^h-\widetilde{U}_1^h)=0
 \quad\text{at }t=0.
\end{equation*}
The coefficient of $\partial_t^2$ is $-1$ and the other terms contain at most one normal derivative. The cited recursion, applied to the difference equation with right-hand side $O(t^\infty)$, therefore gives $\widetilde{U}_2^h-\widetilde{U}_1^h=O(t^\infty)$ for every $h\in\mathcal{H}_\Gamma$. Finally put
\begin{equation*}
 \kappa_0=\frac{\rho_1^{\mathcal{C}}}{\rho_2^{\mathcal{C}}},\qquad
 \lambda_0=\frac{c_1^{\mathcal{C}}}{c_2^{\mathcal{C}}}.
\end{equation*}
Undoing the gauges gives explicitly
\begin{equation*}
 (c_2^{\mathcal{C}}\rho_1^{\mathcal{C}})^{-1}
 (\widetilde{\mathcal{P}}_2^{\mathcal{C}}
   -\widetilde{\mathcal{P}}_1^{\mathcal{C}})\rho_1^{\mathcal{C}}
 =\kappa_0^{-1}\mathcal{P}_2^{\mathcal{C}}\kappa_0
      -\lambda_0\mathcal{P}_1^{\mathcal{C}}=O(t^\infty),
\end{equation*}
which is \eqref{eq:formal-operator-package}. Dividing the solution identity by $\rho_2^{\mathcal{C}}$ gives \eqref{eq:formal-solution-package}. Positivity and the boundary normalizations give \eqref{eq:formal-boundary-normalization}.
\end{proof}

\section{A potential-dependent change of variables}
\label{app:background-dependent-gauge}

We construct the counterexample announced in Remark \ref{rem:background-identification}. It shows why translation covariance restricted to constant boundary data does not fix the spatial identification. The example does not satisfy the full translation covariance used in Proposition \ref{prop:analytic-continuation}.

\begin{proposition}[Constant backgrounds do not fix the identification]\label{prop:background-counterexample}
On the closed unit ball in $\mathbb{R}^n$, $n\geq3$, there exist real analytic currents $\mathcal{J}_1,\mathcal{J}_2$ with the following properties:
\begin{enumerate}[label=(\roman*)]
 \item Both currents vanish at zero gradient and are uniformly elliptic near each constant background. Their chosen constant solutions satisfy $u_m^s\equiv s$, and hence $u_m^{s+t}=u_m^s+t$.
 \item Their nonlinear DN maps agree on a common neighborhood of every constant boundary value. Each current admits a representation $\mathcal{J}_m(x,s,\xi)=\gamma_m(x,s,\xi)\xi$ with $\gamma_m$ real analytic, symmetric, and positive definite for all real $(s,\xi)$.
 \item The identifications of their first linearized Dirichlet solution families are $\Phi_s=F_s^{-1}$, where
 \begin{equation*}
  F_s(x)=x+\varepsilon\tanh(s)(1-|x|^2)e_1,\qquad 0<\varepsilon<\tfrac{1}{4},
 \end{equation*}
 and no single boundary-fixing map gives these identifications for every $s$.
\end{enumerate}
The derivative ellipticity in (i) is asserted only near zero gradient. The full translation covariance in \eqref{condition:elliptic} fails for the second current.
\end{proposition}

\begin{proof}
For a differentiable family of identifications $\Phi_s$, writing $V_s=(\partial_s\Phi_s)\circ\Phi_s^{-1}$ gives
\begin{equation*}
 \partial_s(\Phi_s^*g_{1,s})
 =\Phi_s^*(\partial_sg_{1,s}+\mathcal{L}_{V_s}g_{1,s}),
\end{equation*}
\begin{equation*}
 w_{2,s}^{1,h}
 =\bigl(w_{1,s}^{1,h}+V_sw_{1,s}^h\bigr)\circ\Phi_s.
\end{equation*}
Equality of the two operators after a fixed identification at one background does not imply equality of their background derivatives. The constant-direction part of the second DN derivative is $D^2N_s[1,h]=\partial_s(DN_s[h])$ and does not remove this distinction. General second-order data also involve $\partial_{\xi_j\xi_k}\mathcal{J}^i$. In a gradient-dependent model this tensor can compensate for the moving identification.

This obstruction is realized by an exact invariance of the nonlinear data. In the unit ball, let
\begin{equation*}
 \mathcal{J}_1(\xi)=\frac{\xi}{\sqrt{1+|\xi|^2}},\qquad
 F_s(x)=x+\varepsilon\tanh(s)(1-|x|^2)e_1,
 \quad 0<\varepsilon<\tfrac{1}{4}.
\end{equation*}
Each $F_s$ is a boundary-fixing analytic diffeomorphism. Along a small-gradient solution, set $y=F_{u(x)}(x)$ and $v(y)=u(x)$. At $x=F_s^{-1}(y)$ write
\begin{equation*}
 E=D_xF_s,\quad b=\partial_sF_s,\quad c=\det E>0,
 \quad G=EE^T,\quad d=1-b\cdot\eta,\quad r=\eta^TG\eta.
\end{equation*}
The transformed current is
\begin{equation}\label{eq:potential-dependent-graph-current}
 \mathcal{J}_2(y,s,\eta)=\frac{dG\eta+br}{c\sqrt{d^2+r}}.
\end{equation}
Indeed the Jacobian along the solution graph is $E+b\xi^T$, with $\xi=E^T\eta/d$ and determinant $c/d$. Its Piola transformation gives \eqref{eq:potential-dependent-graph-current} whenever $d>0$. Since the graph map fixes the boundary pointwise, it preserves both the Dirichlet value and the boundary current form. Hence, the two nonlinear DN maps agree on corresponding small neighborhoods of every constant boundary value.

The displayed current extends real analytically to every real $(s,\eta)$ because $d^2+r>0$. Both currents vanish at zero gradient and have positive definite first linearizations there. Their zero-Dirichlet linearizations are invertible by the maximum principle and Schauder theory, so the implicit function theorem gives the local solution branches used above. This is a counterexample within the local elliptic class. No assertion of ellipticity at every real gradient is made.

The obstruction persists even if one requires a symmetric positive definite matrix representation $\mathcal{J}_2=\gamma_2\eta$. Put $C=dI+b\eta^T$ and
\begin{equation*}
 S=\frac{CGC^T}{c\sqrt{d^2+r}},\qquad
 \gamma_2=S+(|\eta|^2I-\eta\eta^T).
\end{equation*}
Then $C^T\eta=\eta$, $S\eta=\mathcal{J}_2$, and $\eta^TS\eta=r/(c\sqrt{d^2+r})>0$ for $\eta\ne0$. The two positive semidefinite summands defining $\gamma_2$ have no common nonzero kernel. At $\eta=0$, $S=G/c>0$. Thus $\gamma_2$ is positive definite and analytic everywhere in the real fiber.

At $s=0$ the two operators and first DN maps equal those of $\Delta$, but with $h=x_1|_{\partial M}$ one has
\begin{equation*}
 w_{1,0}^{1,h}-w_{2,0}^{1,h}
 =\varepsilon(1-|x|^2)\ne0.
\end{equation*}
Here the recovered identifications in \eqref{eq:global-solution-identity} are $\Phi_s=F_s^{-1}$. They cannot be replaced by one fixed map: at $s=0$ the identity of all harmonic solutions forces that map to be the identity, whereas $(F_s^{-1})_1$ varies with $s$. Thus the weaker constant-line hypothesis alone cannot force $\partial_s\Phi_s=0$, even with all nonlinear DN derivatives available.
Since the first current is independent of $s$, its solution map is translation covariant and $w_{1,0}^{1,h}=0$. The displayed nonzero difference therefore gives $w_{2,0}^{1,h}\ne0$, which rules out translation covariance for the second solution map at perturbed data. This completes the counterexample.
\end{proof}

\bigskip\noindent\textbf{Data availability statement.} Data sharing is not applicable to this article, as no datasets were generated or analysed.

\bigskip\noindent\textbf{Conflict of interest.} The author has no conflicts of interest relevant to the content of this article.

\printbibliography{}

\end{document}